\documentclass[a4paper,10pt]{article}

\usepackage[T1]{fontenc}											
\usepackage[utf8]{inputenc}										
\usepackage[italian,english]{babel}							
\usepackage[english]{varioref}									
\usepackage{epigraph}												
\usepackage{xparse}													

\usepackage[sc]{mathpazo}		
\usepackage{eulervm}
\renewcommand{\hbar}{\hslash}		

\usepackage{amsmath}												
\usepackage{amssymb}											
\usepackage{amsthm}												
\usepackage{commath}												
\usepackage{mathtools}											
\usepackage{bigints}													
\usepackage{braket}													
\usepackage{fge}														

\usepackage{graphicx}												
\usepackage{caption}												
\usepackage{float}													
\usepackage{subfig}													
\usepackage{wrapfig}												
\usepackage{pgf}														
\usepackage{tikz}														
	\usetikzlibrary{arrows,calc,intersections}
\usepackage{booktabs}												

\usepackage[marginpar]{todo}

\newtheorem{Theorem}{Theorem}
\newtheorem{Proposition}[Theorem]{Proposition}
\newtheorem{Lemma}[Theorem]{Lemma}

\newtheorem{Remark}[Theorem]{Remark}

\newtheorem{Conjecture}{Conjecture}

\newtheorem{Definition}{Definition}

\RequirePackage{dsfont}

\renewcommand{\subset}{\subseteq}			
\newcommand{\1}{\mathds 1}					

\newcommand{\Z}{\mathbb{Z}}					
\newcommand{\Q}{\mathbb{Q}}					
\newcommand{\R}{\mathbb{R}}					
\newcommand{\Cx}{\mathbb{C}}				

\DeclareMathOperator{\Span}{Span}

\DeclareMathOperator{\SL}{SL}

\DeclareMathOperator{\SU}{SU}

\DeclareMathOperator{\Tan}{T}					
\DeclareMathOperator{\dom}{dom}			

\newcommand{\rec}{\frac{1}}						
\newcommand{\inv}[2][1]{#2^{-#1}}					

\DeclareMathOperator{\ad}{ad}		
\newcommand{\contr}{\! \cdot \!}			

\NewDocumentCommand{\dert}{o m}{\dif \IfNoValueF{#1}{^{#1}}_{#2}}																						
\NewDocumentCommand{\derp}{o m}{\partial\IfNoValueF{#1}{^{#1}}_{#2}}																				
\NewDocumentCommand{\dbydt}{o m G{}}{\frac{\dif \IfNoValueF{#1}{^{#1}} {#3}}{\dif {{#2}}\IfNoValueF{#1}{^{#1}}}}						
\NewDocumentCommand{\dbydp}{o m G{}}{\frac{\partial\IfNoValueF{#1}{^{#1}} {#3}}{\partial { {#2} }\IfNoValueF{#1}{^{#1}}}}	

\let\Re\relax
\DeclareMathOperator{\Re}{Re}

\let\Im\relax
\DeclareMathOperator{\Im}{Im}

\usepackage{calc}

\newsavebox\CBox

\newcommand{\Hilb}{\mathcal H}								
\newcommand{\smth}[1][\infty]{\mathcal C^#1}			
\newcommand{\Schw}{\mathcal S}							
\DeclareMathOperator{\Lint}{L}									

\newcommand{\ModSp}{\mathcal M}							
\RequirePackage{mathrsfs}												
\newcommand{\Lqnt}{\mathscr L}										
\newcommand{\Teich}{\mathcal T}												
\newcommand{\torus}{\mathbb{T}^2}											

\newcommand{\Nabla}{\boldsymbol \nabla}								
\newcommand{\nablatr}{\nabla^{\textup{Tr}}}
\newcommand{\HWC}{\tilde \Nabla{}}										

\newcommand{\opalg}{\mathcal{A}}
\newcommand{\localg}{\opalg_{\text{loc}}}

\newcommand{\opideal}[1][\knot]{\mathcal{I} (#1)}
\newcommand{\locideal}[1][\knot]{\mathcal{I}_{\text{loc}} (#1)}
\newcommand{\AKinv}[1][\knot]{\chi_{#1}}

\newcommand{\teichUhat}{\widehat{U}_{\sigma}}
\newcommand{\teichVhat}{\widehat{V}_{\sigma}}

\newcommand{\teichmhat}{\widehat{m}_{\sigma}}
\newcommand{\teichellhat}{\widehat{\ell}_{\sigma}}

\newcommand{\Uhat}{\widehat{U}}
\newcommand{\Vhat}{\widehat{V}}
\newcommand{\mhat}{\widehat{m}}
\newcommand{\ellhat}{\widehat{\ell}}

\newcommand{\AN}{\mathbb{A}_N}
\newcommand{\ANC}{\AN^{\Cx}}
\newcommand{\xn}{\mathbf{x}}
\newcommand{\yn}{\mathbf{y}}
\newcommand{\qb}{\mathrm{b}}

\newcommand{\cb}{c_{\qb}}
\newcommand{\AKJ}[1][M,K]{J^{(\qb,N)}_{#1}}
\newcommand{\AKJs}[1][\knot]{J_{#1}}
\newcommand{\dilog}{\mathrm{D}_{\qb}}
\newcommand{\philog}{\varphi_{\qb}}

\newcommand{\lx}{\widehat{\ell}_{\xn}}
\newcommand{\mx}{\widehat{m}_{\xn}}
\newcommand{\ly}{\widehat{\ell}_{\yn}}
\newcommand{\my}{\widehat{m}_{\yn}}

\mathtoolsset{showonlyrefs}

\usepackage[bookmarks,colorlinks=false]{hyperref}

\title{A Geometric Quantisation view on\\ the AJ-conjecture for the Teichmüller TQFT}
\author{Jørgen Ellegaard Andersen and Alessandro Malusà}
\date{}

\begin{document} 

\maketitle

\begin{abstract}
	We provide a Geometric Quantisation formulation of the AJ-conjecture for the Teichm\"{u}ller TQFT, and we prove it in detail in the case of the knot complements of $4_{1}$ and $5_2$.
	The conjecture states that the level-$N$ Andersen-Kashaev invariant is annihilated by the non-homogeneous $\hat{A}$-polynomial, evaluated at appropriate $q$-commutative operators.
	We obtained the latter via Geometric Quantisation on the moduli space of flat $\SL(2,\Cx)$-connections on a genus-$1$ surface, by considering the holonomy functions associated to a meridian and longitude.
	The construction depends on a parameter $\sigma$ in the Teichm\"{u}ller space in a way measured by the Hitchin-Witten connection, but we show that the resulting quantum operators are covariantly constant.
	Their action on the Andersen-Kashaev invariant is then defined via a trivialisation of the Hitchin-Witten connection and the Weil-Gel'Fand-Zak transform.
\end{abstract}

\tableofcontents

	\section{Introduction}
		In this paper we consider the $\SL(2,\Cx)$--Chern-Simons theory and the interplay between its formulations via geometric quantisation~\cite{Hit90,ADPW91,Wit89} and the Teichmüller TQFT~\cite{AK14b}.
		Our starting point is the problem of quantising functions on the moduli space of flat connections on a closed oriented surface of genus one, particularly the $A$-polynomial of a knot $K$ embedded in a $3$-dimensional manifold $M$.
		Direct geometric quantisation, however, immediately gives rise to the usual issues.
		First, the pre-quantisation of these functions is in most cases incompatible with the polarisations in use.
		Nonetheless, we find that the holonomy functions associated with a longitude and meridian behave nicely enough to give rise to a pair of $q$-commuting quantum operators $\ellhat$ and $\mhat$.
		Attempting to combine these into a quantisation of polynomial functions would of course cause the usual ordering issues, so we turn to a more indirect approach.
		Namely, we use the so-called Weil-Gel'fand-Zak transform to turn $\ellhat$ and $\mhat$ into operators acting on $\AKJ$, a minor transform of the Andersen-Kashaev partition function on the knot complement $M \setminus K$.
		We are then in the perfect setting to consider the AJ-conjecture~\cite{BDP14,Dim13,Dim15,FGL02,Gar04,Guk05} and give a geometric quantisation formulation of it, specifically for the Teichmüller TQFT.
		Previously proposed for various other versions of the Chern-Simons partition function, the conjecture predicts that the latter is annihilated by an operator which, in the appropriate limit, reproduces the $A$-polynomial.
		This suggests that the desired quantisation should then be obtained as a preferred generator of the annihilator of $\AKJ$ in the algebra of ($q$-commutative) polynomials in $\ellhat$ and $\mhat$.
		We give a precise algebraic definition of such a generator and speculate, following the existing versions of the AJ-conjecture, that the resulting expressions agree with those already found in the literature.
		We carry out the full construction and computation for the first two hyperbolic knots, $4_{1}$ and $5_{2}$, using reduction to find the polynomials.
		We additionally include a detailed account of the convergence of the relevant integrals and a proof that the operators thus constructed annihilate $\AKJ$ as expected.
		
		This work contributes in several ways to the mathematical discussion of quantum $\SL(2,\Cx)$--Chern-Simons theory.
		It brings forth evidence of the little understood relation between two different formulations, extending them both by complementing each other.
		On the one hand, we obtain a precise algebraic statement for the AJ-conjecture within the purely mathematical framework of the Teichmüller TQFT.
		To this, the geometric quantisation side contributes a new interpretation of the $q$-commuting pair, deriving it rigorously and directly from the holonomy functions they are associated with.
		On the other hand, this approach offers candidates for the quantisation of $A$-polynomials as functions on the moduli space of flat connection on a torus, something which geometric quantisation alone does not seem to be able to produce.
		What is more, it strengthens the role of the Weil-Gel'fand-Zak transform, formerly introduced as a bridge between these two approaches.
		While much remains to be understood in those regards, the results of our work show further evidence that this transform will likely play a central role in proving their equivalence.
		
		Let us now provide further background for the sake of context, without attempting to give a detailed historical account.
		
		\subsection{Geometric Quantisation}
		
		The geometric quantisation approach to Chern-Simons theory was proposed by Hitchin~\cite{Hit90} and Axelord-Della Pietra-Witten~\cite{ADPW91} for $\SU(2)$ and by Witten~\cite{Wit91} for $\SL(2,\Cx)$.
		Given a smooth oriented $3$-manifold $X$, possibly with boundary, the classical solutions (i.e. flat connections) form a moduli space which, however, lacks the necessary structure to carry out geometric quantisation.
		For a closed oriented smooth \emph{surface}, on the other hand, the resulting moduli space has a natural symplectic form and pre-quantum line bundle, both tightly related to Chern-Simons theory~\cite{Fre95}.
		This space also comes with a family of polarisations parametrised by the Teichmüller space $\Teich$, but there is no preferred way to choose one in particular.
		Therefore, the quantisation procedures result in vector bundles over $\Teich$, whose fibres are identified (up to rescaling) by the holonomies of appropriate projectively flat connections.
		The latter carry the name of Hitchin and Hitchin-Witten in the respective cases of $\SU(2)$ and $\SL(2,\Cx)$; both have been studied extensively, and they play a pivotal role in Chern-Simons theory, geometric quantisation, and related quantisation schemes~\cite{And05,And06,And10,And12,AG11,AG14,AGL12,AM19,AMR24,AM16,AM23,AN16,Lau10,Mal18,Mar16,Woo92}.
		
		Returning to the geometry of the moduli spaces, the pull-back mapping induced by restricting a connection on $X$ to $\Sigma \coloneqq \partial X$ turns out to be Lagrangian~\cite{Fre95}.
		Said slightly differently, the locus of flat connections on $\Sigma$ which extend to the bulk of $X$ defines a Lagrangian subvariety.
		One may then attempt to quantise this object in place of flat connections on $X$ themselves.
		
		In the setting above, if $X$ is the exterior of an embedded knot $K$ in a \emph{closed} oriented $3$-manifold $M$, then $\Sigma$ is a connected surface of genus one---the boundary of a tubular neighbourhood of $K$.
		The holonomy functions $\ell$ and $m$ associated to a longitude and meridian define global coordinates on the moduli space for $\SL(2,\Cx)$, and (the Zariski closure of) the Lagrangian of interest is cut by a single function $A \in \Z[\ell, m]$---the $A$-polynomial of $K$~\cite{CCGLS94}.
		
		The goal is then to quantise this function $A$ in the case of $\SL(2,\Cx)$.
		As previously mentioned, we do not do this directly, but we start by considering the operators associated to the \emph{logarithmic} holonomy functions $U$ and $V$.
		
		\begin{Theorem}[See Theorems~\ref{thm:teichUVhat}, \ref{thm:normalops}, \ref{thm:nosigmaops}]
			Let $\Teich$ denote the Teichmüller space of a closed oriented surface of genus $1$, and $U$, $V$ the logarithmic holonomies along a meridian and longitude, viewed as functions on the moduli space of flat $\SL(2,\Cx)$-connections.
			For every $\sigma \in \Teich$, the pre-quantum operators defined by $U$ and $V$ preserve the real polarisation $P_{\sigma}$ associated to $\sigma$ as defined by Witten~\cite{Wit91}.
			The resulting \emph{quantum} operators $\teichUhat$ and $\teichVhat$ have central commutator, and moreover they are normal and covariantly constant with respect to the Hitchin-Witten connection.
		\end{Theorem}
		
		This allows us to define operators $\teichmhat$ and $\teichellhat$ as the exponentials of $\teichUhat$ and $\teichVhat$.
		These form a covariantly constant $\Teich$-family of $q$-commutative pairs, each compatible with the respective polarisation.
		We proceed by using a trivialisation of the Hitchin-Witten connection, whose existence and explicit expression was suggested by Witten~\cite{Wit91} and discussed in~\cite{Mal18} (see Section~\ref{subsec:trivialisation})---it was also used in~\cite{AM23} to understand the $\SL(2,\mathbb{C})$ quantum representations of the mapping class groups.
		
		\begin{Theorem}[See Theorem~\ref{thm:nosigmaops}]
			\label{thm:operators_intro}
			The operators $\mhat$ and $\ellhat$, defined by conjugating $\exp(\teichUhat)$ and $\exp(\teichVhat)$ by the trivialisation of the Hitchin-Witten connection, are independent of $\sigma$ and form a $q$-commutative pair for some appropriate $q$.
		\end{Theorem}
		
		\subsection{AJ-conjecture and partition functions}
		
		The AJ-conjecture, in its version for $\SU(2)$, brings together two knot invariants of rather different origin---the $A$-polynomial discussed above and the coloured Jones polynomial.
		Roughly speaking, the conjecture states that the latter is annihilated by an operator $\widehat{A}_{q,K}$, expressed a $q$-commutative polynomial, which in an appropriate sense is minimal and reproduces the $A$-polynomial when $q = 1$.
		Early works on this include, in no particular order, those of Frohman-Gelca-Lofaro~\cite{FGL02}, Garoufalidis~\cite{Gar04}, and~\cite{Guk05}.
		In~\cite{FGL02}, the authors construct a non-commutative analogue of the $A$-polynomial by replacing the coordinate rings of character varieties with Kauffman bracket skein modules.
		They show their invariant is orthogonal to the coloured Jones polynomial under the natural pairing, and speculate whether its containment in the annihilator is proper.
		In~\cite{Gar04}, Garoufalidis studies the coloured Jones polynomial in terms of recursive relations, and formulates his version of the conjecture in terms of a preferred generator of the ideal of such relations, a viewpoint further developed in~\cite{GL05,GL16,GLL18,Le06}.
		The approach proposed by Gukov~\cite{Guk05} and later expanded e.g. in~\cite{BDP14,DGG14,Dim13,Dim15,DGLZ09}, on the other hand, is in terms of the partition function of (analytically continued) Chern-Simons theory.
		In that context, the equation $A(\ell, m) = 0$ represents the classical configurations (i.e. flat connections) on a knot complement $M$ as sitting inside those on the boundary torus.
		Analogously, Gukov views the partition function $Z$ of $M$ as the wave function of a state on the boundary, which one should expect to satisfy a relation of the form $\widehat{A} Z = 0$ quantising the classical equation $A(\ell, m) = 0$.
		The relation with the coloured Jones naturally follows from Witten's interpretation of it as the partition function for the $\SU(2)$-theory~\cite{Wit89}.
		
		Although Witten's work on the coloured Jones polynomial was originally based on path-integral methods, his results were later made mathematically rigorous by Reshetikhin-Turaev~\cite{RT90,RT91} in terms of category theory and TQFT's.
		The equivalence of this partition-function approach to geometric quantisation was later established by a chain of isomorphisms due to Andersen-Ueno~\cite{AU1,AU2,AU3,AU4} and Laszlo~\cite{Las98}.
		
		For non-compact groups, such as $\SL(2,\Cx)$ and $\SL(2,\R)$, the situation is more involved.
		Numerous constructions of partition functions have been proposed with varying degrees of generality, using both path-integral techniques and more mathematically rigorous methods.
		Without discussing specific details, some examples in no particular order include~\cite{AK14a,AK14b,AK14c,BB04,BB07,BDP14,DFM11,Dim13,Dim15,DGLZ09,Guk05,Hik01,Hik07}.
		Most of these approaches are based on Faddeev's quantum dilogarithm (\textsection~\ref{subsec:ANdilog}) and some version of gluing formula, and in the examples that have been explicitly carried out they seem to essentially agree.
		Despite the similarities, however, the exact relation between all these different approaches, as well as with geometric quantisation, is not yet fully understood, to the best of these authors' knowledge.
		Nonetheless, Gukov's argument for the AJ-conjecture is general enough that it should apply to any partition-function model of Chern-Simons theory, regardless of the specifics of its construction.
		
		Of particular inspiration for our work are the methods of Dimofte~\cite{Dim13} and Beem-Dimofte-Pasquetti~\cite{BDP14}, based on ideal triangulations of knot complements.
		Besides producing their own version of the partition function, they also propose a plan for quantising the $A$-polynomial directly from gluing data.
		On the classical side, to the gluing of ideal tetrahedra corresponds symplectic reduction on the space of flat connections on the boundary, and the equation of the Lagrangian associated to the bordism can be obtained from those of the individual tetrahedra by an elimination and evaluation process.
		On the other hand, recognising the Faddeev difference operator~\eqref{eq:FaddeevEqn} as a quantisation of the Lagrangian of a single tetrahedron, Dimofte proposes using $q$-commutative elimination theory to mimic the same process as a quantisation of the $A$-polynomial.
		The construction is carried out explicitly in the cited works for some of the first few knots, and the appropriate conjectures are verified there for those cases.
		
		Our next goal is then to use similar elimination techniques and the operators discussed above to find relations on the partition function of the Teichmüller TQFT.
		For context, the latter is expressed as a functor between suitably defined categoroids of decorated $(2+1)$-cobordisms and infinite-rank topological vector spaces.
		After its first formulation in~\cite{AK14a}, dependent of a single quantum parameter $\qb$, it was further extended in~\cite{AK14b} to include a second parameter $N$, called its level---see~\cite{AM16} for more details.
		What is most important to know for this work is that the theory assigns a complex number, dependent of the pair $(\qb,N)$ of quantum parameters, to every hyperbolic knot $K$ embedded in a closed oriented $3$-manifold $M$.
		This object---the partition function of the theory---is conjecturally equivalent to the aforementioned $\AKJ$, which is a complex-valued entire function on the space $\AN \coloneqq \R \times \Z \slash N\Z$.
		Importantly, $\AKJ$ is expected to share several fundamental properties with the coloured Jones polynomial, further strengthening its role as an $\SL(2,\Cx)$ analogue of it.
		The precise statements are discussed in detail in op.~cit.; the relevant conjectures are verified therein for the knot complements of $4_{1}$ and $5_{2}$, for $6_{1}$ (with $N = 1$) by Andersen-Nissen~\cite{AN16}.
		Later works of Ben Aribi-Piquet-Nakazawa~\cite{BP19} and Ben Aribi-Guéritaud-Piguet-Nakazawa~\cite{BGP23} extended their results to all twist knots.
		
		The key ingredient in reconnecting this viewpoint to that of geometric quantisation is the so-called Weil-Gel'fand-Zak transform~\cite{AK14b}.
		This map identifies Schwartz-class functions on $\AN$ with smooth sections of the level-$N$ Chern-Simons line bundle for a genus-one closed oriented surface.
		As it happens, the transform is also compatible with the natural $\Lint^{2}$ norms on the two spaces, and therefore extends to a unitary isomorphism between their closures.
		In particular, we can use this to translate our $\mhat$ and $\ellhat$ from Theorem~\ref{thm:operators_intro} into operators acting on functions on $\AN$.
		
		\begin{Theorem}[See Lemma~\ref{lemma:conjugation}, Theorem~\ref{thm:conjugation}]
			Under conjugation by the Weil-Gel'fand-Zak transform, $\mhat$ and $\ellhat$ correspond to the operators $\mx$ and $\lx$ defined by
			\begin{equation}
				\begin{gathered}
					\bigl( \mx f \bigl) (x,n) \coloneqq e^{-2\pi \frac{\qb x}{\sqrt{N}}} e^{2\pi i \frac{n}{N}} f(x,n) \, , \\
					\bigl( \lx f \bigr) (x,n) \coloneqq f \biggl( x - \frac{i\qb}{\sqrt{N}}, n+1 \biggr) \, .
				\end{gathered}
			\end{equation}
		\end{Theorem}
		
		Crucially, $\mx$ and $\lx$ are the fundamental constituents of the Faddeev difference operator~\eqref{eq:FaddeevEqn}.
		They are also the same operators used by Dimofte~\cite{Dim13,Dim15} as the building blocks for the construction of his version of the quantum $\widehat{A}$-polynomial.
		
		In order to formulate our conjecture precisely, we shall take an algebraic approach similar to that of~\cite{Gar04}.
		Namely, we consider the algebra $\mathcal{A}$ formally generated over $\Z[q^{\pm \frac{1}{2}}]$ by two $q$-commuting elements $E$ and $Q$.
		The operators $\lx$ and $\mx$ then define a representation of $\mathcal{A}$, and we may consider the annihilator of $\AKJ$.
		If this left ideal is non-trivial it contains a preferred generator (in a sense discussed below), which we shall call the $\widehat{A}^{\Cx}$-polynomial.
		
		\begin{Conjecture}
			\label{conj:complexAJ}
			Let $K \subset M$ be a hyperbolic knot inside a closed oriented $3$-manifold.
			Then the ideal annihilating $\AKJ$ is non-trivial.
			The resulting $\widehat{A}^{\Cx}$-polynomial agrees with the non-homogeneous $\widehat{A}$-polynomial from the coloured Jones theory (see \textsection~\ref{sec:alg_setup}) up to a right factor in $Q$, and it reproduces the classical $A$-polynomial in the evaluation at $q = 1$, $E = \ell$, and $Q = m^{2}$, again up to a factor in $m$.
		\end{Conjecture}
		
		As highlighted above, our statement is very close in spirit to the works of Gukov~\cite{Guk05} and Dimofte~\cite{Dim13,Dim15}.
		In fact, in the final part of this article we will use the same $q$-commutative elimination process as Dimofte to produce a guess for the $\widehat{A}^{\Cx}$-polynomial.
		In contrast to his \emph{a priori} approach, however, we will work indirectly from the explicit expression of $\AKJ$.
		Since we are using the same operators and the partition functions match, the algebraic manipulations will result in the same non-commutative polynomials, which do agree with the non-homogeneous $\widehat{A}$ from the coloured Jones theory.
		In order to conclude that the polynomials thus obtained annihilate $\AKJ$ as expected, we will need to carefully carry out some further analytic checks.
		After doing so, we can conclude the following.
		
		\begin{Theorem}
			\label{thm:main_intro}
			Conjecture~\ref{conj:complexAJ} holds for the figure-eight knot $4_{1}$, and $5_{2}$.
		\end{Theorem}
		
		To give a sense of the procedure, for the two knots in question and $\xn \in \AN$ the function $\AKJ$ takes the form
		\begin{equation}
			\AKJ(\xn) = \int_{\AN} \Phi (\xn, \yn) \dif \yn
		\end{equation}
		for some meromorphic function $\Phi$ defined on $\ANC \times \ANC$, $\ANC \coloneqq \Cx \oplus \Z \slash N\Z$.
		To this corresponds an annihilator in the obvious representation of $\localg^{\otimes 2}$, and its structure makes it straightforward to find generators of this ideal.
		It is intuitively clear here that the action of $\mx$ and $\lx$ should commute with integration, while $\ly$ essentially amounts to a change of variable.
		It should then follow that any element in $\localg^{\otimes 2}$ which annihilates $\Phi$ and does not contain $\my$ will also annihilate $\AKJ$ after evaluation at $\ly = 1$.
		Once again, this is consistent with the elimination/evaluation procedure of Dimofte, although differently motivated.
		
		Upon closer inspection on $\Phi$ and its integral, however, one sees that its convergence may not hold when $\xn$ or the integration contour are shifted in the imaginary direction.
		To wok around this, we will introduce new integration contours $\gamma_{h,a}$, labelled by appropriate parameters and stable under shifts.
		We will establish convergence of the new integral whenever $\xn$ lies in appropriate regions $R_{h,a}$, and that the sum is holomorphic and agrees with $\AKJ$ for real $\xn$.
		This gives a full characterisation of the holomorphic extension of $\AKJ$ to the whole $\ANC$.
		Choosing the parameters $h$ and $a$ appropriately, the region $R_{h,a}$ will then be large enough to be stable under any set number of shifts in both variables.
		Given a polynomial in $\lx$, $\mx$, and $\ly$ annihilating the integrand, this will allow us to take each individual monomial out of the integral (after replacing $\ly$ with $1$),  thus obtain an operator which kills the partition function.
		
		\subsection*{Structure of the paper}
		
		In section~\ref{sec:background} we give an overview of the background material we refer to throughout the the rest of the work.
		This includes generalities on geometric quantisation and the Hitchin-Witten connection, the level-$N$ quantum dilogarithm and the Weil-Gel'fand-Zak transform, the Teichmüller TQFT, and some algebraic setup.
		In section~\ref{sec:setting} we define the precise structure to which we are going to apply geometric quantisation, and argue that it provides a model for (a double cover of) the moduli space relevant for genus one Chern-Simons theory.
		In section~\ref{sec:geom_quant} we actually run the geometric quantisation machinery to obtain the desired operators.
		First, we use the standard definition of the pre-quantum operators to quantise the logarithmic holonomy functions corresponding to the meridian and longitude on the torus.
		Next, we check that the operators are compatible with the chosen polarisation, thus descending to \emph{quantum} operators.
		We then show that these are normal, thus admitting well-defined exponentials.
		Finally, we use the explicit trivialisation of the Hitchin-Witten connection to remove the dependence of the operators on the Teichmüller parameter.
		After that, we determine the action of the operators on functions on $\AN$ via the Weil-Gel'fand-Zak transform.
		In section~\ref{sec:AK}, we explicitly carry out the procedure described above to find the $\widehat{A}^{\Cx}$-polynomial for the first two hyperbolic knots.
		
	\subsection*{Acknowledgements}
		
		We wish to thank Tudor Dimofte for profitable insight on the techniques used for computing the polynomial annihilating $J^{(\qb,N)}_{M,K}$, and Simone Marzioni for frequent discussion on aspects of his work on the Teichmüller TQFT.
		We also owe an acknowledgement to the developers of Singular~\cite{DGPS}, which we have used along the process of computing the polynomials.
		
		Both authors were supported in part by the Danish National Science Foundation Center of Excellence grant, Centre for Quantum Geometry of Moduli spaces, DNRF95.
		
		The first-named author is supported by the grant from the Simons foundation, Simons Collaboration on New Structures in Low-Dimensional Topology grant no. 994320, the ERC-SyG project, Recursive and Exact New Quantum Theory (ReNewQuantum) with funding from the European Research Council (ERC) under the European Union‘s Horizon 2020 research and innovation programme, grant agreement no. 810573.
		
		The second-named author, in addition, thanks for their support the University of Toronto and the University of Saskatchewan, as well as the Pacific Institute for the Mathematical Sciences (PIMS) and the Centre for Quantum Topology and its Applications (quanTA).

	\section{General background}
		\label{sec:background}
		
	\subsection{\texorpdfstring{$\AN$}{AN} and the level-\texorpdfstring{$N$}{N} quantum dilogarithm}
		\label{subsec:ANdilog}

		\begin{Definition}
			For every positive integer $N$, let $\AN$ be the locally compact Abelian group $\R \oplus \Z/N\Z$ endowed with the nomalised Haar measure $\dif \, (x,n)$ defined by
		\begin{equation}
			\label{eq:measure}
			\int_{\AN} f(x,n) \dif{} (x,n) := \rec{\sqrt{N}} \sum_{n = 1}^{N} \int_{\R} f(x,n) \, \dif x \, .
		\end{equation}
		We denote by $\Schw (\AN , \Cx)$ the space of Schwartz class functions on $\AN$, i.e. functions $f(x,n)$ on $\AN$ which restrict to Schwartz class functions on $\R$ for every $n$.
		We shall denote $\Cx \oplus \Z/N\Z$ by $\ANC$.
		\end{Definition}
		
		Of course $\Schw (\AN , \Cx)$ sits inside the space $\Lint^{2} (\AN , \Cx)$ of square-summable functions, as a dense subspace.
		We will often use the notation $\xn = (x,n)$; moreover, if $\lambda \in \Cx$ we write $\xn + \lambda$ as a short-hand for $(x+\lambda , n)$.
		
		As in~\cite{AK14b}, we use the following notations for Fourier Kernels and Gaussians on $\AN$:
		\begin{equation}
			\label{eq:FourierGaussian}
			\begin{gathered}
				\Braket{(x,n),(y,m)} = e^{ 2 \pi i xy} e^{- 2\pi i nm/N} \, , \\
				\Braket{(x,n)} = e^{\pi i x^2} e^{- \pi i n (n + N)/N} \, .
			\end{gathered}
		\end{equation}
		
		Fix now $\qb$, a complex unitary parameter with $\Re(\qb) > 0$ and $\Im(\qb) \geq 0$, and introduce constants
		\begin{equation}
			\cb := \frac{ i ( \qb + \inv{\qb} ) }{2} = i \Re(\qb) \, ,
				\qquad
			q^{\rec{2}} := - e^{ \pi i \frac{b^2 + 1}{N}} = \inv{ \Bigl\langle \! \Bigl( \frac{i \qb}{\sqrt{N}} , -1 \Bigr) \! \Bigr\rangle } \, .
		\end{equation}
		
		We summarise here the fundamental properties of the level $N$ quantum dilogarithm which are relevant for this work.
		For the precise definition and further details see e.g.~\cite{AK14b,AM16}.
		For $N$ a positive \emph{odd} integer, the quantum dilogarithm $\dilog$ at level $N$ and quantum parameter $\qb$ is a meromorphic function on $\ANC$ which satisfies the Faddeev difference equations
		\begin{equation}
			\label{eq:FaddeevEqn}
			\begin{aligned}
				\dilog \Bigl( x \pm \frac{i\qb}{\sqrt{N}} , n \pm 1 \Bigr) =&\ {\left( 1 - e^{\pm \frac{ \qb^{2} + 1}{N}} e^{ 2 \pi \frac{\qb}{\sqrt{N}} x} e^{2 \pi i \frac{n}{N}} \right)}^{\mp 1} \dilog (x,n) \, ,	\\[.5em]
				\dilog \Bigl( x \pm \frac{i \overline{\qb}}{\sqrt{N}} , n \mp 1 \Bigr) =&\ {\left( 1 - e^{\pm \frac{ \overline{\qb}^{2} + 1}{N}} e^{ 2 \pi \frac{\overline{\qb}}{\sqrt{N}} x} e^{-2 \pi i \frac{n}{N}} \right)}^{\mp 1} \dilog (x,n)
			\end{aligned}
		\end{equation}
		and the the inversion relation
		\begin{equation}
			\dilog (\xn) \dilog (-\xn) = \inv{\zeta_{N,\textup{inv}}} \Braket{\xn} \, , \qquad \zeta_{N,\textup{inv}} = e^{\pi i ( N + 2 \cb^2 \inv{N})/6} \, .
		\end{equation}
		
		\begin{Lemma}
			\label{lemma:dilog_asymptotics}
			For $n \in \Z/N\Z$ fixed, the quantum dilogarithm has the following asymptotic behaviour for $x \to \infty$:
			\begin{equation}
				\dilog(x,n) \approx
				\begin{cases}
					1
						& \text{on $\abs{\arg(x)} > \frac{\pi}{2} + \arg(\qb)$,}	\\[.5em]
					\inv{\zeta_{N,\textup{inv}}} \Braket{\xn}
						& \text{on $\abs{\arg(x)} < \frac{\pi}{2} - \arg(\qb)$.}
				\end{cases}
			\end{equation}
			Furthermore, the dilogarithm satisfies the unitarity relation
			\begin{equation}
				\overline{\dilog(x,n)} \dilog(\overline{x},n) = 1 \, .
			\end{equation}
		\end{Lemma}
	
		It is convenient to change the notation according to~\cite{AM16}, calling
		\begin{equation}
			\philog (x,n) : = \dilog (x,-n) \, .
		\end{equation}
		The zeroes and poles of $\philog$ occur at the points $\mathbf{p}_{\alpha,\beta}$ and $- \mathbf{p}_{\alpha,\beta}$ respectively, for $\alpha,\beta \in \Z_{\geq 0}$, where
		\begin{equation}
			\mathbf{p}_{\alpha,\beta} := \Bigl( - \frac{ \cb + i \alpha \qb + i \beta \overline{\qb} }{\sqrt{N}} \, , \, \alpha - \beta \Bigr) \, .
		\end{equation}
		We shall often call
		\begin{equation}
			\label{eq:T}
			T \coloneqq
			\begin{cases}
				\Set{ x \in \Cx \colon
					\Re \Bigl(\overline{b} \Bigl(x + \frac{\cb}{\sqrt{N}} \Bigr) \Bigr) \leq 0 \text{ and }
					\Re \Bigl(b \Bigl( x + \frac{\cb}{\sqrt{N}} \Bigr) \Bigr) \geq 0 }
					&\text{if $\qb \ne 1$,}	\\
				\Set{ x \in \Cx \colon \Im(x) \leq 1 \text{ and } \abs{\Re(x)} \leq 1}
					&\text{if $\qb = 1$.}
			\end{cases}
		\end{equation}
		In particular, the zeroes and poles of $\philog(x,n)$ for $n$ fixed occur only for $x \in T$ and $x \in -T$ respectively.
		Lemma~\ref{lemma:dilog_asymptotics} holds unchanged for $\philog$ in place of $\dilog$.
		
		\begin{Definition}
			\label{def:ANops}
			If $k \in \Z_{>0}$ and $\boldsymbol{\mu} : (\ANC)^{k} \to \ANC$ is a $\Z$-linear function, denote by $\widehat{m}_{\boldsymbol{\mu}}$ the operator acting on complex-valued functions on $(\ANC)^{k}$ as
			\begin{equation}
				\widehat{m}_{\boldsymbol{\mu}} f := \Bigl\langle \boldsymbol{\mu} , \Bigl( \frac{i \qb}{\sqrt{N}} , -1 \Bigr) \Bigr\rangle f \, .
			\end{equation}
			Moreover, call $\lx$ the operator acting on complex-valued functions on $\ANC$ as
			\begin{equation}
				\bigl(\lx f\bigr) (x,n) := f \Bigl( x - \frac{i \qb}{\sqrt{N}} , n+1 \Bigr) \, .
			\end{equation}
		\end{Definition}
		
		\begin{Remark}
			\label{rem:operators_domain}
			The action of $\mx$ and $\lx$ is clearly well defined on meromorphic functions, and in fact on (dense subspaces of) $\Lint^{2}(\AN, \Cx)$, with the \emph{caveat} that $i \qb \slash \sqrt{N}$ cannot be real, owing to the condition that $\Re(\qb) > 0$.
			On the one hand, this implies that the factor $\langle \xn, (i\qb/\sqrt{N}, -1) \rangle$ is unbounded on $\AN$ as it grows exponentially for $x \to -\infty$.
			Nonetheless, the domain of $\mx$ as an operator on $\Lint^{2} (\AN, \Cx)$ contains all compactly supported functions, and is therefore dense.
			On the other hand, the shift along $(-i\qb/\sqrt{N}, 1)$ does not preserve $\AN \subset \ANC$, so that, strictly speaking, $\lx$ is not well defined on $\Lint^{2} (\AN, \Cx)$.
			However, if $f \in \Lint^{2}(\AN, \Cx)$ is entire, i.e. analytic with infinite radius of convergence, then it has a unique holomorphic extension, and there is a natural way to make sense of $\lx f$.
			This occurs, for instance, whenever $f$ is the Fourier transform of a compactly supported function, in which case $\lx f$ is also square-integrable, showing that $\lx$ is also a densely defined operator on $\Lint^{2} (\AN, \Cx)$.
			In fact, a little Fourier analysis also shows that, for any $\lambda \in \Cx$, (the closure of) the shift operator along $(\lambda, 0)$ is the exponential of $\lambda \dbydt{x}$, a fact we shall use later.
		\end{Remark}
		
		The following lemma is an immediate consequence of the definitions and Faddeev's difference equation.
		
		\begin{Lemma}
			\label{lemma:opGaussDilog}
			The operator $\lx$ acts on the Gaussian and the quantum dilogarithm as
			\begin{equation}
				\begin{gathered}
				\lx \Braket{\xn} = q^{-\rec{2}} \mx^{-1} \Braket{\xn} \, , \\
					\lx \philog (\xn) = \Bigl( 1 + q^{-\rec{2}} \inv{\mx} \Bigr) \philog (\xn) \, , \\
					\inv{\lx} \philog (\xn) = \Bigl( 1 + q^{\rec{2}} \inv{\mx} \Bigr)^{-1} \philog (\xn) \, .
				\end{gathered}
			\end{equation}
			Moreover, $\lx$ and $\mx$ undergo the commutation relation
			\begin{equation}
				\lx \mx = q \mx \lx \, .
			\end{equation}
		\end{Lemma}
		
	\subsection{Algebraic setting and \texorpdfstring{$A$}{A}-polynomials}
		\label{sec:alg_setup}
		
		For the algebraic setup of our conjecture, we shall borrow notations from Garoufalidis~\cite{Gar04} as follows.
	
		First, we consider the $q$-commutative algebra
		\begin{equation}
			\opalg := \Z[q^{\pm \frac{1}{2}}] \langle E , Q \rangle \Big/ \bigl( EQ - qQE \bigr) \, .
		\end{equation}
		One can also make sense of inverting polynomials in $Q$ and obtain
		\begin{equation}
			\label{eq:localg}
			\localg := \Set{ \sum_{k=0}^{l} a_{k} (q,Q) E^{k} \colon l \in \Z_{\geq 0} , \, a_{k} \in \Q(q,Q)}
		\end{equation}
		with the product determined by
		\begin{equation}
			\bigl( a (q,Q) E^{k} \bigr) \cdot \bigl( b(q,Q) E^{h} \bigr) := a(q,Q) b(q,q^{k}Q) E^{k+h} \, .
		\end{equation}
		
		Given a representation of $\localg$ and a vector $e$ in it, one may define the sets
		\begin{equation}
			\locideal[e] \coloneqq \Set{ p(Q,E) \in \localg \colon p(Q,E) e = 0}
			\qquad \text{and} \qquad
			\opideal[e] \coloneqq \locideal[e] \cap \opalg
		\end{equation}
		of elements annihilating $e$, each of which is a left ideal in its respective algebra.
		Since every ideal in $\localg$ is principal, there exists a unique generator of $\locideal[e]$ of minimal degree in $E$ and co-prime coefficients in $\Z[q,Q]$.
		As such, this element actually lies in $\opalg$, thus giving a preferred ``generator'' of $\opideal[e]$.
		
		One way to phrase the AJ-conjecture for the coloured Jones polynomial is to study recursive relations on it in terms of representations of $\opalg$ and $\localg$.
		This leads to a definition of the $\widehat{A}$-polynomial as the preferred generator of $\opideal[K] \coloneqq \opideal[J_{K}]$ as discussed above.
		One version of this construction leads to the so-called non-homogeneous polynomial $\widehat{A}_{q,K} (Q,E)$.
		We shall later refer to the formul\ae{} found in~\cite{GS10} for $4_{1}$ and $5_{2}$, which read
		\begin{equation}
			\label{eq:Ahat41_known}
			\begin{split}
				\widehat{A}_{q,4_{1}}^{\text{nh}} ={}& q^{2} \bigl( q^{2} Q - 1 \bigr) \bigl( q Q^{2} - 1\bigr) Q^{2} E^{2} \\
					& - \bigl( q Q - 1 \bigr) \bigl( qQ + 1 \bigr) \Bigl( q^{4} Q^{4} - q^{3} Q^{3} - q (q^2 + 1) Q^{2} - q Q + 1 \Bigr) E \\
					& + q^{2} \bigl( Q - 1 \bigr) \bigl( q^{3} Q^{2} - 1) Q^{2} \, ,
			\end{split}
		\end{equation}
		\begin{equation}
			\label{eq:Ahat52}
			\begin{split}
				\widehat{A}_{q,5_{2}}^{\text{nh}} \!={}& \big( q^3 Q - 1\big) \big(q Q^2 - 1\big) \big(q^2 Q^2 - 1\big) E^3 \\
					& + q \bigl(q^{2} Q - 1\bigr) \bigl( q Q^{2} - 1 \bigr) \bigl(q^{4} Q^{2} - 1 \bigr) \\
					& \contr \Bigl( q^9 Q^5 - q^7 Q^4 - q^4 (q^3 - q^2 - q + 1) Q^3 + q^2(q^3 + 1) Q^2 + 2q^2 Q - 1\Big) E^2 \\
					& - q^{5} Q^{2} \bigl(q Q - 1 \bigr) \bigl( q^{2} Q^{2} - 1 \bigr) \bigl(q^{5} Q^{2} - 1 \bigr) \\
					& \contr \Bigl( q^{6} Q^{5} - 2 q^{5} Q^{4} - q^{2} (q^{3} + 1) Q^{3} + q (q^{3} - q^{2} - q +1) Q^{2} + q Q - 1 \Bigr) E \\
					& + q^{9} Q^{7} \bigl( Q - 1 \bigr) \bigl( q^{4} Q^{2} - 1 \bigr) \bigl( q^{5} Q^{2} - 1 \bigr) \, .
			\end{split}
		\end{equation}
		The $A$-polynomials of these knots, known to be irreducible~\cite{HS04}, read
		\begin{equation}
			\label{eq:clA41}
			A_{4_{1}} (m , \ell) = m^{4} \ell^{2} - \Bigl( m^{8} - m^{6} - 2 m^{4} - m^{2} + 1 \Bigr) \ell + m^{4} \, ,
		\end{equation}
		\begin{equation}
			\label{eq:clA52}
			\begin{split}
				A_{5_{2}} (m , \ell) ={}& \ell^{3} + \Bigl( m^{10} - m^{8} + 2 m^{4} + 2 m^{2} - 1 \Bigr) \ell^{2} \\
					& - m^{4} \Bigl( m^{10} - 2 m^{8} - 2 m^{6} + m^{2} - 1 \Bigr) \ell + m^{14} \, .
			\end{split}
		\end{equation}
		
	\subsection{The Andersen-Kashaev theory}
		\def\knot{f}
		
		The Andersen-Kashaev theory defines an infinite-rank TQFT $\mathcal{Z}$ from quantum Teichmüller theory.
		In particular, it defines an invariant $\mathcal{Z} (X)$ for every object $X$ consisting of a closed oriented $3$-manifold $M$, a hyperbolic knot $K$, and a suitably decorated triangulation of its complement.
		However, it is conjectured~\cite{AK14a,AM16} that a two-parameter family of smooth functions $J^{(\qb,N)}_{M,K} (\xn)$ on $\AN$ exists such that
		\begin{equation}
			\mathcal{Z}_{\qb}^{(N)} (X) = e^{i \cb^{2} \phi} \int_{\AN} J^{(\qb,N)}_{M,K} (\xn) e^{i \lambda \cb x} \dif \xn \, ,
		\end{equation}
		where $\lambda$ and $\phi$ carry the information relative to the decorated triangulation, while $J^{(\qb,N)}_{M,K} (\xn)$ depends on the pair $(M,K)$ alone.
		In addition, this function is also conjectured to enjoy certain asymptotic conditions analogous to those expected from the coloured Jones polynomial, something that has been established in several particular cases~\cite{AN16,BP19,BGP23}.
		For the knots $4_{1}$ and $5_{2}$ in $S^{3}$, the expression for $J^{(\qb,N)}_{M,K}$ is found to be
		\begin{equation}
			\begin{gathered}
				J^{(\qb,N)}_{S^{3} \!,\, 4_{1}} (\xn) = e^{4 \pi i \frac{\cb x}{\sqrt{N}}} \AKinv[4_{1}] (\xn) \, , 
					\qquad
				\AKinv[4_{1}] (\xn) = \int_{\AN} \frac{ \philog (\xn - \yn) {\Braket{\yn}}^{2}}{\philog(\yn) {\Braket{\xn-\yn}}^{2}} \, \dif \yn \, , \\[1em]
				J^{(\qb,N)}_{S^{3} \!, \, 5_{2}} (\xn) = e^{2 \pi i \frac{\cb x}{\sqrt{N}}} \AKinv[5_{2}] (\xn) \, , 
					\qquad
				\AKinv[5_{2}] (\xn) = \int_{\AN} \frac{\Braket{\yn} \inv{\Braket{\xn}}}{\philog(\yn + \xn) \philog(\yn) \philog(\yn-\xn)} \, \dif \yn \, .
			\end{gathered}
		\end{equation}

\section{Setup for geometric quantisation}
	\label{sec:setting}

		In this section we shall introduce the space on which we will run geometric quantisation, as well as all the relevant notations and conventions.
		A more general version of this discussion may be found in~\cite{AMR22}.

		Throughout this paper, we will write $\torus$ to denote the real torus $S^1 \times S^1$, and $\torus_{\Cx}$ for the $2$-dimensional complex torus $\Cx^{*} \times \Cx^{*}$ containing it.
		We shall use coordinates $u,v \in \R$ on $\torus$ and $U,V \in \Cx$ on $\torus_{\Cx}$, with
		\begin{equation}
			(u,v) \mapsto \big(e^{2 \pi i u} , e^{2 \pi i v} \big) \in \torus \, ,
				\qquad
			(U,V) \mapsto \big( e^{2 \pi i U} , e^{2 \pi i V} \big) \in \torus_{\Cx} \, .
		\end{equation}
		We refer to these as the logarithmic coordinates, as opposed to the exponential coordinates
		\begin{equation}
			m = e^{2 \pi i U} \, ,
				\qquad
			\ell = e^{2 \pi i V}
		\end{equation}
		on $\torus_{\Cx}$.
		Using $u = \Re(U)$, $v = \Re(V)$, $\Im(U)$, and $\Im(V)$ as real coordinates, it makes sense to regard $\dbydp{u}$ and $\dbydp{v}$ as vector field on $\torus_{\Cx}$ as well as on $\torus$.
		
		On these spaces we consider symplectic $2$-forms
		\begin{equation}
			\omega = - 2 \pi \dif u \wedge \dif v \, ;
				\qquad
			\omega_{\Cx} = - 2 \pi \dif U \wedge \dif V \, .
		\end{equation}
		We now fix a positive integer $N$ and a real number $S$ (without further restrictions), calling $t = N + iS$ the level of the theory, and define the  level-$t$ \emph{real} symplectic structure on $\torus_{\Cx}$ as
		\begin{equation}
			\omega_{t} \coloneqq \rec{2} \Re \big( t \omega_{\Cx} \big) \, ,
		\end{equation}
		which restricts to the form $N \omega$ on $\torus$.
		A pre-quantum line bundle $\Lqnt^{(t)}$ on $(\torus_{\Cx} , \omega_{t})$ is defined by the quasi-periodicity conditions
		\begin{equation}
				\psi(U+1 , V) = e^{- \pi i \Re(t V)} \psi (U,V) \, ,
				\quad \text{and} \quad
				\psi(U , V+1) = e^{\pi i \Re(t U)} \psi (U,V) \, ,
		\end{equation}
		and connection $\nabla^{(t)} = \dif - i \theta^{(t)}$ with
		\begin{equation}
			\theta^{(t)}_{(U,V)} = \pi \Re \Big( t \big( V \dif U - U \dif V \big) \Big) \, .
		\end{equation}
		This bundle restricts to one on $(\torus,N\omega)$, which we call $\Lqnt^{N}$.
		Explicitly, the quasi-periodicity conditions and the connection form for this bundle are
		\begin{equation}
			\begin{gathered}
				\psi(u+1 , v) = e^{- N\pi i v} \psi (u,v)
				\quad \text{and} \quad
				\psi(u , v+1) = e^{N \pi i u} \psi (u,v) \, , \\
				\theta^{N}_{(u,v)} = N \pi \bigl( v \dif u - u \dif v \bigl) \, .
			\end{gathered}
		\end{equation}
		We refer to these as the Chern-Simons line bundles over $\torus_{\Cx}$ and $\torus$ at the level $t$ and $N$ respectively, and we shall often omit the superscript in the connection.
		
	\paragraph{The family of complex structures on \texorpdfstring{$\torus$}{T}.}
		Denote by $\Teich$ the upper half-plane
		\begin{equation}
			\Teich = \Set{ \sigma \in \Cx : \Im(\sigma) > 0} \, .
		\end{equation}
		To every point of $\Teich$ one can associate an almost complex structure on $\torus$ represented in the logarithmic coordinates by the constant matrix
		\begin{equation}
			\label{eq:complex_tensor}
			J := \frac{i}{\sigma - \overline{\sigma}}
			\begin{pmatrix}
				- (\sigma + \overline{\sigma})			&	2 \sigma \overline{\sigma}	\\[.5em]
				-2													&	\sigma + \overline{\sigma}
			\end{pmatrix} \, .
		\end{equation}
		It is easily checked that this defines a complex structure on $\torus$, with holomorphic and anti-holomorphic vector fields given by
		\begin{equation}
			\label{eq:holo_fields}
			\begin{gathered}
				\dbydp{\overline{w}} : = \frac{\1 + i {J}}{2} \dbydp{u}
					= \rec{\sigma - \overline{\sigma}} \left( \sigma \dbydp{u} + \dbydp{v} \right) \, , \\
				\dbydp{w} : = \frac{\1 -i {J}}{2} \dbydp{u}
					= - \rec{\sigma - \overline{\sigma}} \left( \overline{\sigma} \dbydp{u} + \dbydp{v} \right) \, ,
			\end{gathered}
		\end{equation}
		to which correspond complex coordinates
		\begin{equation}
			w = u - \sigma v \, ,
			\qquad
			\overline{w} = u - \overline{\sigma} v \, .
		\end{equation}
		For later convenience, note that $\omega$ is determined in these coordinates by
		\begin{equation}
			\omega \left( \dbydp w , \dbydp {\overline w} \right) = \rec 4 \omega \left( \dbydp u - i J \dbydp u , \dbydp u + i J \dbydp u \right) = - \frac{2 \pi}{\sigma - \overline{\sigma}} \, ,
		\end{equation}
		which implies
		\begin{equation}
			\label{eq:omegaholo}
			\bigl[ \nabla_{w} , \nabla_{\overline{w}} \bigr] = \frac{2 N \pi i}{\sigma - \overline{\sigma}} \, .
		\end{equation}
		
		Together with $\omega$, $J$ defines a Kähler structure on $\torus$ with metric
		\begin{equation}
			g =
			\frac{2 \pi i}{\sigma - \overline{\sigma}}
			\begin{pmatrix}
				2													&	- (\sigma + \overline{\sigma})			\\
				- (\sigma + \overline{\sigma})			&	2 \sigma \overline{\sigma}
			\end{pmatrix} \, ,
		\end{equation}
		whose inverse is
		\begin{equation}
			\tilde{g} = \frac{i}{2 \pi ( \sigma - \overline{\sigma})}
			\begin{pmatrix}
				2 \sigma \overline{\sigma}	&	\sigma + \overline{\sigma}	\\
				\sigma + \overline{\sigma}	&	2
			\end{pmatrix} \, .
		\end{equation}
		The Laplace operator $\Delta$, which acts on sections of $\Lqnt^{N}$ by differentiating twice and then contracting both indices with $\tilde{g}$, can be written as
		\begin{equation}
			\label{eq:laplacian}
			\Delta = - i \frac{\sigma - \overline{\sigma}}{2 \pi} \Bigl( \nabla_{w} \nabla_{\overline{w}} + \nabla_{\overline{w}} \nabla_{w} \Bigr) 
			= -i \frac{\sigma - \overline{\sigma}}{\pi} \nabla_{w} \nabla_{\overline{w}} - N \, ,
		\end{equation}
		by noticing that the metric is determined by
		\begin{equation}
			g \left( \dbydp w , \dbydp {\overline w} \right) = - i \omega \left( \dbydp w , \dbydp {\overline w} \right) = \frac{2 \pi i}{\sigma - \overline{\sigma}} \, .
		\end{equation}
		Since the coefficients of $g$ are constant functions on $\torus$, its Levi-Civita connection is trivial for all values of $\sigma$, and therefore independent of it.
		Consequently, the variation of $\Delta$ with respect to $\sigma$ is determined by that of $\tilde{g}$, which is
		\begin{equation}
			\dbydp{\sigma}{\tilde{g}} = \frac{i}{\pi (\sigma - \overline{\sigma})^{2}}
			\begin{pmatrix}
				\overline{\sigma}^{2}	&	\overline{\sigma}	\\
				\overline{\sigma}			&	1
			\end{pmatrix} \, ,
			\qquad
			\dbydp{\overline{\sigma}}{\tilde{g}} = - \frac{i}{\pi (\sigma - \overline{\sigma})^{2}}
			\begin{pmatrix}
				\sigma^{2}	&	\sigma	\\
				\sigma			&	1
			\end{pmatrix} \, .
		\end{equation}
		These two tensors are also parallel, and up to constant coefficients one can recognise them as $\dbydp{w} \otimes \dbydp{w}$ and $\dbydp{\overline{w}} \otimes \dbydp{\overline{w}}$.
		In particular, this makes $J$ holomorphic and rigid in the sense of~\cite{AG14}.
		The variation of $\Delta$ is then
		\begin{equation}
			\dbydp{\sigma} \Delta = - \frac{i}{\pi} \nabla_{w} \nabla_{w} \, ,
			\qquad
			\dbydp{\overline{\sigma}} \Delta = \frac{i}{\pi} \nabla_{\overline{w}} \nabla_{\overline{w}} \, .
		\end{equation}
		
		\paragraph{Polarisations on \texorpdfstring{$\torus_{\Cx}$}{TC} and geometric quantisation.}
		Using the natural complex structure $I$ on $\torus_{\Cx}$, the right-most expressions in~\eqref{eq:holo_fields} may also be read as \emph{real} vector fields on the complex torus.
		To avoid confusion, we shall denote these as $\overline{X}$ and $X$, respectively, although both objects are real and not the conjugate of one another.
		They span integrable distributions in $\Tan \torus_{\Cx}$ which are Lagrangian for $\omega_{\Cx}$, thus for $\omega_{t}$ for every $t$, i.e. polarisations.
		We set
		\begin{equation}
			P = P_{\sigma} := \Span \bigl\langle \overline{X}, I \overline{X} \bigr\rangle \, .
		\end{equation}
		Because each leaf of $P$ intersects $\torus \subset \torus_{\Cx}$ at exactly one point, and transversely, this subspace may be identified with the reduction $\torus_{\Cx}/P$.
		One can identify the space of smooth polarised sections of $\Lqnt^{(t)}$ over $\torus_{\Cx}$ with that of \emph{all} smooth sections of $\Lqnt^{N}$ over $\torus$, the latter supporting an $\Lint^{2}$-product via the volume form $\omega$.
		In other words, one can define the level-$t$ Hilbert space $\Hilb_{\sigma}^{(t)}$ arising from geometric quantisation on $(\torus_{\Cx} , \omega_{t})$ with pre-quantum line bundle $\Lqnt^{(t)}$ and polarisation $P_{\sigma}$ as $\Lint^{2} (\torus , \Lqnt^{N})$.
		Although $\sigma$ does not manifestly enter the definition of this last space, the dependence on this parameter should be measured via the Hitchin-Witten connection~\cite{Wit91,AG14} on the trivial bundle
		\begin{equation}
			\Hilb^{t} := \Teich \times \Lint^{2} (\torus , \Lqnt^{N}) \to \Teich \, .
		\end{equation}
		Due to the flatness of $g$, the definition of the connection simplifies to
		\begin{equation}
			\HWC_{\sigma} = \dbydp{\sigma} + \frac{i}{\pi} \nabla_{w} \nabla_{w} \, ,
			\qquad
			\HWC_{\overline{\sigma}} = \dbydp{\overline{\sigma}} - \frac{i}{\pi} \nabla_{\overline{w}} \nabla_{\overline{w}} \, .
		\end{equation}
		Although the arguments of~\cite{AG14} do not apply, $\torus$ having non-trivial holomorphic vector fields and fist cohomology, it is not difficult to show that $\HWC$ is flat in this case.
		In fact, Witten proposes the following statement.
		
		\begin{Proposition}
			\label{prop:HWCtrivialisation}
			The Hitchin-Witten connection $\HWC$ for $\torus_{\Cx}$ has a trivialisation
			\begin{equation}
				\HWC = \exp(-r \Delta) \nablatr \exp(r \Delta)
			\end{equation}
			for $r$ a complex parameter such that
			\begin{equation}
				e^{4Nr} = - \frac{\overline{t}}{t} \, .
			\end{equation}
		\end{Proposition}
		The result can be proven via a straightforward adaptation of the argument presented in~\cite{AM19}.
		
	\paragraph{Motivation: the moduli spaces of flat connections on a genus-one surface.}
		\label{subsec:genus1}
				
			The definitions introduced in this section are motivated by the $\SL(2,\Cx)$--Chern-Simons theory on a smooth oriented surface $\Sigma$ of genus one.
			If $G$ denotes either $\SU(2)$ or $\SL(2,\Cx)$, the moduli space of flat $G$-connections on $\Sigma$ can be realised as a product of two copies of a maximal torus in $G$, modulo the action of the Weyl group $W \simeq \Z/2\Z$.
			The moduli spaces, which we denote as $\ModSp$ and $\ModSp_{\Cx}$, can then be described as $\torus/W$ and $\torus_{\Cx}/W$ respectively, where $W$ acts on each space by simultaneously inverting both entries.
			One can use the coordinates above on the moduli spaces, on which are defined the respective Atiyah-Bott forms $\omega^{\text{AB}}$ and $\omega_{\Cx}^{\text{AB}}$, which pull back to $2 \omega$ and $2 \omega_{\Cx}$.
			For every positive integer $k$ and real number $s$, Chern-Simons theory defines pre-quantum line bundles $\Lqnt^{k}$ and $\Lqnt^{(k+is)}$ for $k \omega^{\text{AB}}$ and $\omega_{k+is}^{\text{AB}} = \Re((k+is) \omega_{\Cx}^{\text{AB}})$.
			It follows from the definitions that these lift to $\Lqnt^{2k}$ and $\Lqnt^{(2k + 2is)}$ on $\torus$ and $\torus_{\Cx}$.
			
			If $(x,y)$ are $1$-periodic coordinates on the surface, every $\sigma \in \Teich$ defines a Riemann surface structure on $\Sigma$ with holomorphic coordinate $z = x + \inv{\sigma} y$ (for the \emph{reversed} orientation).
			This correspondence gives a biholomorphism between $\Teich$ and the Teichmüller space of $\Sigma$.
			The Hodge $*$-operator, which defines the Kähler structure on $\ModSp$ for the given Riemann surface structure, is represented in these coordinates by the matrix $J$ of~\eqref{eq:complex_tensor}.
			Vectors on the moduli spaces are identified with Lie-algebra valued forms on $\Sigma$: if $T$ is a generator of a Cartan sub-algebra of $\mathfrak{su}(2)$, to $\dbydp{u}$ and $\dbydp{v}$ correspond $T \dif x$ and $T \dif y$.
			
			In order to run geometric quantisation, Witten defines a polarisation on $\ModSp_{\Cx}$ spanned by the forms of type $(1,0)$; since $T \dif z$ represents $\dbydp{\overline{w}}$ up to rescaling, this lifts to $P$ on $\torus_{\Cx}$.
			Therefore, the quantum Hilbert space thus obtained for the $\SL(2,\Cx)$--Chern-Simons theory at the level $k+is$ is contained in $\Hilb_{\sigma}^{(t)}$ for $t = 2(k+is)$, as the sub-space of $W$-invariant sections.
			As $\sigma$ varies, these spaces form a sub-bundle of $\Hilb^{(t)}$, identified with $\Teich \times \Lint^{2} ( \ModSp , \Lqnt^{k})$, which is preserved by $\HWC$.
			The restriction is the connection introduced by Witten in~\cite{Wit89}.
		
	\section{Operators from geometric quantisation on \texorpdfstring{$\torus_{\Cx}$}{T2}}
	\label{sec:geom_quant}
	\subsection{The quantum operators on \texorpdfstring{$\Hilb^{(t)}_{\sigma}$}{Hsigma}}
	
		We now fix $t = N + iS$ and study the level-$t$ pre-quantum operators associated to the logarithmic coordinates $U$ and $V$ on $\torus_{\Cx}$.
		Strictly speaking, these functions are only well defined up to picking a branch, so we should start by specifying one.
		For instance, we may choose the coordinates on $\torus$ so that $0 \leq u, v < 1$ and impose that $U$ and $V$ extend them continuously away from the $P_{\sigma}$-leaves through $\set{u = 0}$ and $\set{v = 0}$, respectively.
		It will be clear later that the resulting operators are essentially independent of the choice of a specific branch (see Remark~\ref{rem:branches}).
		We shall also talk freely of the differentials and Hamiltonian vector fields of $U$ and $V$ regardless of their discontinuity, since these objects extend unambiguously to their singular locus.
		
		\begin{Theorem}
			\label{thm:teichUVhat}
			For every $\sigma \in \Teich$, the pre-quantum operators of $U$ and $V$ are compatible with the polarisation $P_{\sigma}$ and therefore descend to the \emph{quantum} Hilbert space.
			Their action on smooth sections of $\Lqnt^{N}$ over $\torus$ is given by
			\begin{equation}
				\teichUhat \coloneqq u - \frac{i \sigma}{\pi t} \nabla_{w} \, ,
					\qquad
				\teichVhat \coloneqq v - \frac{i}{\pi t} \nabla_{w} \, .
			\end{equation}
		\end{Theorem}
		
		\begin{proof}
			Recall that the pre-quantum operator of a function $f$ on $\torus_{\Cx}$ is defined on sections of $\Lqnt^{(t)}$ as
			\begin{equation}
				\widehat{f} \coloneqq f - i \nabla_{H_{f}} \, ,
			\end{equation}
			where $H_{f}$, the Hamiltonian vector field of $f$ relative to $\omega_{t}$, is determined by
			\begin{equation}
				\label{eq:condition}
				Y[f] = \omega_{t} (Y, H_{f}) \quad \text{for every $Y \in \Tan \torus_{\Cx}$.}
			\end{equation}
			It is well known that $\widehat{f}$ preserves the space of polarised sections if and only if the Lie derivative by $H_{f}$ preserves the space of vector fields tangent to $P$.
			This is clearly the case for $U$ and $V$, given that both the symplectic form and the generators of $P$ have constant coefficients.
			
			For the last part of our assertion it is enough to show that $\frac{\sigma}{\pi t} \dbydp{w}$ and $\frac{1}{\pi t} \dbydp{w}$ differ from $H_{U}$ and $H_{V}$ by elements of $P$.
			By direct calculation we see
			\begin{equation}
				\begin{gathered}
					\omega_{t} \biggl( \overline{X}, \dbydp{w} \biggr) = \frac{ \pi t}{\sigma - \overline{\sigma}} = \frac{\pi t}{\sigma} \overline{X}[U] = \pi t \overline{X}[V] \, ,
					\\
					\omega_{t} \biggl( I\overline{X}, \dbydp{w} \biggr) = \frac{i\pi t}{\sigma - \overline{\sigma}} = \frac{\pi t}{\sigma} (I\overline{X})[U] = \pi t (I\overline{X})[V] \, .
				\end{gathered}
			\end{equation}
			Since $\overline{X}$ and $I\overline{X}$ span $P$, it follows from~\eqref{eq:condition} that $H_{U} - \frac{\sigma}{\pi t} \dbydp{w}$ and $H_{V} - \frac{1}{\pi t} \dbydp{w}$ are $\omega_{t}$-orthogonal to $P$, and our conclusion follows.
		\end{proof}
		
		We now wish to define quantum operators for the exponential coordinates $m$ and $\ell$, to which end we rely on the spectral theorem for normal operators, see e.g.~\cite{Con94}.
		In summary, a densely defined operator $E$ on a separable Hilbert space is called normal if it is closed, shares the same domain as its adjoint $E^{\dagger}$, and the two commute.
		The spectral theorem states that any such operator is unitarily equivalent to the multiplication by a function $\phi$ on the $\Lint^{2}$ space of some measure space.
		The exponential $\exp(E)$ is then defined as the operator corresponding to $e^{\phi}$, also closed and densely defined.
		
		\begin{Remark}
			\label{rem:spectral_thm}
			In the following we shall use two consequences of the spectral theorem for a normal operator $E$ on a separable Hilbert space.
			First, if $\psi$ is a vector on which the exponential series of $E$ converges, then the sum equals $\exp(E) \psi$.
			Second, there exists a nested family of subspaces $H_{C}$, $C \in \R_{>0}$ whose union is dense and such that, for each $C$, both $E$ and $E^{\dagger}$ preserve $H_{C}$ and are bounded by $C$ on that subspace.
			In particular the exponential series of $E$ is strongly convergent on every $H_{C}$, and therefore $\exp(E)$ may be expressed as a series on a dense subspace.
		\end{Remark}
		
		\begin{Theorem}
			\label{thm:normalops}
			The quantum operators $\teichUhat$ and $\teichVhat$, acting on $\Hilb^{(t)}_{\sigma}$, are normal.
		\end{Theorem}
		
		\begin{proof}
			On the one hand, $u$ and $v$ are bounded and self-adjoint, so the conditions on the domains of $\teichUhat$ and $\teichVhat$ break down to $\nabla_{w}$.
			The latter operator is well defined on the subspace $W^{1,2} (\torus,\Lqnt^{N}) \subset \Lint^{2} (\torus, \Lqnt^{N})$ consisting of all sections whose (distributional) covariant derivatives along $\dbydp{u}$ and $\dbydp{v}$ are themselves $\Lint^{2}$-sections.
			A standard exercise shows that, with this domain, $\nabla_{w}$ is a closed operator with adjoint $-\nabla_{\overline{w}}$ defined on the same domain.
			
			We check the commutation relations by direct computation, namely
			\begin{equation}
				\begin{split}
					 \Bigl[ \teichUhat , \teichUhat^{\dagger} \Bigr] ={}& \Bigl[ u - \frac{i \sigma}{\pi t} \nabla_{w} , u - \frac{i \overline{\sigma}}{\pi \overline{t}} \nabla_{\overline{w}} \Bigr] = \\
						={}& \frac{i \sigma}{\pi t} \cdot \frac{\overline{\sigma}}{\sigma - \overline{\sigma}} + \frac{i \overline{\sigma}}{\pi \overline{t}} \cdot \frac{\sigma}{\sigma - \overline{\sigma}}
							- \frac{\sigma \overline{\sigma}}{ \pi^{2} t \overline{t}} \cdot \frac{2 N \pi i}{\sigma - \overline{\sigma}} = 0
				\end{split}
			\end{equation}
			and
			\begin{equation}
				\begin{split}
					 \Bigl[ \teichVhat , \teichVhat^{\dagger} \Bigr] ={}& \Bigl[ v - \frac{i}{\pi t} \nabla_{w} , v - \frac{i}{\pi \overline{t}} \nabla_{\overline{w}} \Bigr] = \\
						={}& \frac{i}{\pi t} \cdot \frac{\overline{\sigma}}{\sigma - \overline{\sigma}} + \frac{i}{\pi \overline{t}} \cdot \frac{\sigma}{\sigma - \overline{\sigma}}
							- \frac{1}{ \pi^{2} t \overline{t}} \cdot \frac{2 N \pi i}{\sigma - \overline{\sigma}} = 0 \, .
						\qedhere
				\end{split}
			\end{equation}
		\end{proof}

		Since for every $\lambda \in \Cx$ and every normal operator $N$ on a Hilbert space $\lambda N$ is also normal, the lemma ensures then that the following is well posed.
		
		\begin{Definition}
			We define quantum operators associated to $m$ and $\ell$ on $\Hilb^{(t)}_{\sigma}$ as
			\begin{equation}
				\teichmhat = \exp \Bigl( 2 \pi i \teichUhat \Bigr) \, ,
					\qquad
				\teichellhat = \exp \Bigl( 2 \pi i \teichVhat \Bigr) \, .
			\end{equation}
		\end{Definition}
		
		\begin{Remark}
			\label{rem:branches}
			A different branch of $U$, say continuous on an open dense, would differ from the first by a function $c$ valued in $\Z$, and thus locally constant.
			This change is of no consequence on $H_{U}$, and therefore the only effect on $\teichUhat$ is to add $c$.
			However, the multiplication by a locally constant function commutes with all differential operators, and therefore by Baker-Campbell-Hausdorff we have
			\begin{equation}
				\exp \Bigl( 2 \pi i \bigl(\teichUhat + c \bigr) \Bigr) = e^{2 \pi i c} \exp \Bigl( 2 \pi i \teichUhat \Bigr) = \exp \Bigl( 2 \pi i \teichUhat \Bigr) \, .
			\end{equation}
			In other words, $\teichmhat$ is unaffected by choosing a different branch, and the situation is analogous for $V$ and $\teichellhat$.
		\end{Remark}
		
	\subsection[Trivialisation of the Hitchin-Witten connection]{Trivialisation of the Hitchin-Witten connection and \texorpdfstring{$\sigma$}{sigma}-independent operators}
		\label{subsec:trivialisation}
		
		Our next goal is to show that, after trivialising $\HWC$ using Proposition~\ref{prop:HWCtrivialisation}, the operators from the previous section become $\sigma$-independent.
		
		\begin{Definition}
			We define the $\sigma$-independent quantum operators of $U$, $V$, $m$, and $\ell$ as
			\begin{equation}
				\begin{array}{cc}
					\Uhat \coloneqq \exp \bigl(r \Delta \bigr) \teichUhat \exp \bigl(-r\Delta \bigr) \, ,
						\qquad &
					\Vhat \coloneqq \exp \bigl(r \Delta \bigr) \teichVhat \exp \bigl(-r\Delta \bigr) \, , \\[.4em]
					\mhat \coloneqq \exp \bigl(r \Delta \bigr) \teichmhat \exp \bigl(-r\Delta \bigr) \, ,
						\qquad &
					\ellhat \coloneqq \exp \bigl(r \Delta \bigr) \teichellhat  \exp \bigl(-r\Delta \bigr) \, .
				\end{array}
			\end{equation}
		\end{Definition}
		
		The phrasing of the definition above is justified by the following result.
		
		\begin{Theorem}
			\label{thm:nosigmaops}
			The $\sigma$-independent operators are
			\begin{equation}
				\begin{gathered}
					\Uhat = u - i \frac{e^{2rN} - 1}{2 N \pi} \nabla_{v}
						\qquad
					\Vhat = v + i \frac{e^{2rN} - 1}{2 N \pi} \nabla_{u} \, , \\[.4em]
					\mhat = \exp \bigl( 2 \pi i \Uhat \bigr) = e^{2 \pi i u} \exp \biggl( \frac{e^{2rN} - 1}{N} \nabla_{v} \biggr) \, , \\[.4em]
					\ellhat = \exp \bigl( 2 \pi i \Vhat \bigr) = e^{2 \pi i v} \exp \biggl( - \frac{e^{2rN} - 1}{N} \nabla_{u} \biggr) \, ,
				\end{gathered}
			\end{equation}
			and therefore are indeed independent of $\sigma$.
		\end{Theorem}
		
		\begin{proof}
			We proceed to study $\Uhat$ and $\Vhat$ by expanding $\exp(\pm r\Delta)$ as power series.
			Throughout the proof we will use that
			\begin{equation}
				\label{eq:ad}
				\begin{aligned}
					[\Delta, u] = - \frac{i}{\pi} \bigl( \sigma \nabla_{w} - \overline{\sigma} \nabla_{\overline{w}} \bigr) \, , &\qquad
					[\Delta, v] = - \frac{i}{\pi} \bigl( \nabla_{w} - \nabla_{\overline{w}} \bigr) \, , \\
					[\Delta, \nabla_{w}] = - 2N \nabla_{w} \, , &\qquad
					[\Delta, \nabla_{\overline{w}}] = 2N \nabla_{\overline{w}} \, ,
				\end{aligned}
			\end{equation}
			from which it follows by induction that
			\begin{equation}
				\begin{gathered}
					\Delta^{n} \teichUhat = \teichUhat \Delta^{n} - \frac{i}{\pi t}
						\sum_{k = 1}^{n} \binom{n}{k} (2N)^{k-1} \Bigl( (-1)^{k} \sigma \overline{t} \nabla_{w} - \overline{\sigma} t \nabla_{\overline{w}} \Bigr) \Delta^{n-k} \, , \\
					\Delta^{n} \teichVhat = \teichVhat \Delta^{n} - \frac{i}{\pi t}
						\sum_{k = 1}^{n} \binom{n}{k} (2N)^{k-1} \Bigl( (-1)^{k} \overline{t} \nabla_{w} - t \nabla_{\overline{w}} \Bigr) \Delta^{n-k} \, .
				\end{gathered}
			\end{equation}
			
			To prove the theorem, suppose that subspaces $H_{C} \subset \Lint^{2} (\torus, \Lqnt^{N})$, $C \in \R_{>0}$ are given as in Remark~\ref{rem:spectral_thm}, for $E = \Delta$.
			We will show below that the series
			\begin{equation}
				\label{eq:totconv}
				S \psi \coloneqq \sum_{n,m \in \Z_{\geq 0}} \frac{(-1)^{m} r^{n+m}}{n!m!} \Delta^{n} \teichUhat \Delta^{m} \psi \, .
			\end{equation}
			is totally convergent whenever $\psi$ lies in $H_{C}$.
			Assuming this as a given for now, we see on the one hand, summing over $n$ first and then over $m$, that
			\begin{equation}
				S\psi = \sum_{m = 0}^{\infty} \frac{(-r)^{m}}{m!} \exp(r\Delta) \teichUhat \Delta^{m} \psi 
				= \exp(r\Delta) \teichUhat \biggl( \sum_{m = 0}^{\infty} \frac{(-r)^{m}}{m!} \Delta^{m} \psi \biggr)
				= \Uhat \psi
			\end{equation}
			where we used that $\exp(r\Delta)$ is continuous and $\teichUhat$ closed.
			On the other hand, a different arrangement of the terms yields
			\begin{equation}
				S \psi = \sum_{n = 0}^{\infty} \sum_{k = 0}^{n} \frac{(-1)^{n-k} r^{n}}{k!(n-k)!} \Delta^{k} \teichUhat \Delta^{n-k} \psi
				= \sum_{n = 0}^{\infty} \frac{r^{n}}{n!} \ad_{\Delta}^{n} (\teichUhat) \psi \, ,
			\end{equation}
			which using~\eqref{eq:ad} evaluates to
			\begin{equation}
				\begin{split}
					S \psi
					={}& \Bigl(u - \frac{i \sigma}{\pi t} \nabla_{w} \Bigr) \psi
						-\frac{i}{2Nt\pi} \sum_{n = 1}^{\infty} \frac{(2Nr)^{n}}{n!} \bigl( (-1)^{n} \sigma \overline{t} \nabla_{w} - \overline{\sigma} t \nabla_{\overline{w}} \bigr) \psi =\\
					={}& u \psi
						+ \frac{i}{\pi} \biggl(\Bigl( - \frac{\overline{t}}{t} \frac{e^{-2Nr} -1}{2N} - \frac{1}{t} \Bigl) \sigma \nabla_{w}
						+ \frac{e^{2Nr}-1}{2N} \overline{\sigma} \nabla_{\overline{w}} \biggr) \psi = \\
					={}& u \psi + \frac{i (e^{2Nr}-1)}{2N\pi} \bigl(\sigma \nabla_{w} + \overline{\sigma} \nabla_{\overline{w}}\bigr) \psi
					= \Bigl( u - \frac{i (e^{2Nr} - 1)}{2N\pi} \nabla_{v} \Bigr) \psi \, .
				\end{split}
			\end{equation}
			This establishes the desired equality for $\Uhat$ on $H_{C}$ for every $C$, and thus on a dense subspace.
			Since the operators are closed, the equality then extends to the respective domains.
			
			What remains to be seen is the total convergence of~\eqref{eq:totconv}.
			It is well known that $\Delta$ is a self-adjoint operator with essential domain consisting of all $\Lint^{2}$ sections whose weak Laplacian is itself an $\Lint^{2}$-section.
			Therefore, if $\psi \in \dom(\Delta)$ is approximated by a sequence of smooth sections $\psi_{n}$, then $\Delta \psi_{n}$ converges in $\Lint^{2}$ (to $\Delta \psi$).
			Given any $\varepsilon > 0$, using~\eqref{eq:laplacian} we find
			\begin{equation}
				\begin{split}
					\norm{\nabla_{\overline{w}}(\psi_{n}-\psi_{m})}^{2} ={}& \abs{\braket{\nabla_{w} \nabla_{\overline{w}} (\psi_{n} - \psi_{m}), \psi_{n}-\psi_{m}}} \\
					\leq{} & \frac{\pi}{\abs{\sigma-\overline{\sigma}}} \norm{(\Delta+N) (\psi_{n} - \psi_{m})} \norm{\psi_{n} - \psi_{m}} < \varepsilon
				\end{split}
			\end{equation}
			for $n$ and $m$ sufficiently large.
			Therefore, $\nabla_{\overline{w}} \psi_{n}$ is a Cauchy sequence in $\Lint^{2}$ and therefore $\psi \in \dom(\nabla_{\overline{w}}) = \dom(\nabla_{w})$.
			If, in particular, $\psi \in H_{C}$, then a similar manipulation yields
			\begin{equation}
				\norm{\nabla_{\overline{w}} \psi}^{2} \leq \frac{C+N}{\abs{\sigma - \overline{\sigma}}} \norm{\psi}^{2} \eqqcolon R^{2} \norm{\psi}^{2} \, ,
			\end{equation}
			and similarly for $\nabla_{w} \psi$.
			Since $\Delta$ preserves $H_{C}$, the same will hold with $\Delta^{n} \psi$ in place of $\psi$ for any $n$.
			Moreover, using the expressions for $\teichUhat$ and $\teichVhat$ in Theorem~\ref{thm:teichUVhat} similar inequalities will hold for these operators as well.
			
			For every $n,m \in \Z_{\geq 0}$ we then have that
			\begin{equation}
				\begin{split}
					\norm{\Delta^{n} \teichUhat \Delta^{m} \psi}
					\leq{} & \bigl( 1 + \abs{\sigma} R \bigr) C^{n+m} \norm{\psi}
						+ \sum_{k = 1}^{n} \binom{n}{k} (2N)^{k-1} \abs{\sigma} R C^{n-k+m} \norm{\psi} \\
					\leq{} & C^{m} \biggl(C^{n} + \abs{\sigma} R \sum_{k = 0}^{n} \binom{n}{k} (2N)^{k} C^{n-k} \biggr) \norm{\psi} \\
					={}& C^{m} \biggl(C^{n} + \abs{\sigma} R (C+2N)^{n} \biggr) \norm{\psi} \, .
				\end{split}
			\end{equation}
			This is enough to show total convergence of~\eqref{eq:totconv} as claimed, and finally establish our claim on $\Uhat$.
			
			The process for $\Vhat$ is completely analogous.
			The relations for $\mhat$ and $\ellhat$ follow since exponentiation is stable under conjugation by unitary maps, the splitting following by Baker-Campbell-Hausdorff.
		\end{proof}
		
	\subsection{The Weil-Gel'fand-Zak transform}
		
		\begin{Lemma}[\cite{AK14b},\cite{AM16}]
			The map $W^{(N)} \!\colon\! \Schw (\AN , \Cx) \to \smth (\torus , \mathcal{L}^N)$ defined by
			\begin{equation}
				f(x,n)		\mapsto	s(u,v) = e^{i\pi N uv} \sum_{m \in \Z} f \Big( \sqrt{N} u + \frac{m}{\sqrt{N}} , -m \Big) e^{2 \pi i mv}
			\end{equation}
			is an isomorphism.
			Moreover, it intertwines the $\Lint^{2}$-pairings on the two spaces and thus extends to an isometry of their completions.
		\end{Lemma}
		
		The above map is called the Weil-Gel'fand-Zak transform, and it transforms the quantum operators on $\Hilb^{(t)}$ according to the following statement.
				
		\begin{Lemma}
			\label{lemma:conjugation}
			For every $f \in \Schw (\AN, \Cx)$, one has
			\begin{equation}
				\begin{gathered}
					\nabla_{u} W^{(N)} \big( f(\xn) \big) = W^{(N)} \bigl(\sqrt{N} f' (\xn) \bigr) \, , \\
					\nabla_{v} W^{(N)} \big( f(\xn) \big) = W^{(N)} \Big(2\pi i \sqrt{N} x f(\xn) \Big) \, , \\
					e^{2 \pi i u} W^{(N)} \big( f(\xn) \big) = W^{(N)} \Big( e^{2 \pi i \frac{x}{\sqrt{N}}} e^{2 \pi i \frac{n}{N}} f(\xn) \Big) \, , \\
					e^{2 \pi i v} W^{(N)} \big( f(\xn) \big) = W^{(N)}  \Bigl( f \Big( x - \rec{\sqrt{N}} , n + 1 \Big) \Bigr) \, .
				\end{gathered}
			\end{equation}
		\end{Lemma}
		
		\begin{proof}
			We proceed by direct computation.
			Fast decay of Schwartz-class functions and their derivatives justifies term-by-term differentiation, which yields
			\begin{equation}
				\begin{split}
					\nabla_{u} W^{(N)} \big( f(\xn) \big) ={}& \dbydp{u} W^{(N)} \big( f(\xn) \big) - i \pi N v W^{(N)} \big( f(\xn) \big) = \\
					={}& e^{i\pi N uv} \sum_{m \in \Z} \dbydp{u} f \Big( \sqrt{N} u + \frac{m}{\sqrt{N}} , -m \Big) e^{2 \pi i mv} + \\
					& + i \pi N v W^{(N)} \big( f \big) - i \pi N v W^{(N)} \big( f \big) = \\
					={}& e^{i\pi N uv} \sum_{m \in \Z} \sqrt{N} f' \Big( \sqrt{N} u + \frac{m}{\sqrt{N}} , -m \Big) e^{2 \pi i mv} \, .
				\end{split}
			\end{equation}
			Similarly, differentiation in $v$ yields
			\begin{equation}
				\begin{split}
					\nabla_{v} W^{(N)} \big( f (\xn) \big) ={}& \dbydp{v} W^{(N)} \big( f(\xn) \big) + N \pi i u W^{(N)} \big( f(\xn) \big) = \\
					={}& N \pi i u e^{N \pi i uv} \sum_{m \in \Z} f \Big( \sqrt{N} u + \frac{m}{\sqrt{N}} , -m \Big) e^{2 \pi i mv} + \\
					& 2 \pi i e^{N \pi i uv} \sum_{m \in \Z} m f \Big( \sqrt{N} u + \frac{m}{\sqrt{N}} , -m \Big) e^{2 \pi i mv} + \\
					& + N \pi i u W^{(N)} \big( f(x,n) \big) = \\
					= 2 \pi i \sqrt{N} & e^{N \pi i uv} \sum_{m \in \Z} \Big(\sqrt{N} u + \frac{m}{\sqrt{N}} \Big) f \Big( \sqrt{N} u + \frac{m}{\sqrt{N}} , -m \Big) e^{2 \pi i mv} \, .
				\end{split}
			\end{equation}
			By a simple manipulation we see that
			\begin{equation}
				\begin{split}
					e^{2 \pi i u} W^{(N)} \big( f(\xn) \big) ={}& e^{i\pi N uv} \sum_{m \in \Z} e^{2 \pi i u} f \Big( \sqrt{N} u + \frac{m}{\sqrt{N}} , -m \Big) e^{2 \pi i mv} = \\
					={}& e^{i\pi N uv} \sum_{m \in \Z} e^{2 \pi i ( u + \frac{m}{N})} e^{- 2 \pi i \frac{m}{N}} f \Big( \sqrt{N} u + \frac{m}{\sqrt{N}} , -m \Big) e^{2 \pi i mv} \, .
				\end{split}
			\end{equation}
			Finally, changing variable from $m$ to $m-1$ we find
			\begin{equation}
				\begin{split}
					e^{2 \pi i v} W^{(N)} \big( f(\xn) \big) ={}& e^{i\pi N uv} \sum_{m \in \Z} f \Big( \sqrt{N} u + \frac{m}{\sqrt{N}} , -m \Big) e^{2 \pi i (m+1) v} = \\
					={}& e^{N \pi i uv} \sum_{m \in \Z} f \Big( \sqrt{N} u + \frac{m-1}{\sqrt{N}} , -m + 1 \Big) e^{-2 \pi i mv} \, .
					\qedhere
				\end{split}
			\end{equation}
		\end{proof}
		
		\begin{Theorem}
			\label{thm:conjugation}
			Let $t = N+iS$ be fixed, $r$ as in Proposition~\ref{prop:HWCtrivialisation}, $\qb \coloneqq -ie^{2rN}$.
			Then the Weil-Gel'fand-Zak transform intertwines the operators $\mhat$ and $\ellhat$ on $\Lint^{2} (\torus, \Lqnt^{N})$ with $\mx$ and $\lx$ (cf. Definition~\ref{def:ANops}) on $\Lint^{2} (\AN, \Cx)$, respectively.
		\end{Theorem}
		
		\begin{proof}
			The identities of Lemma~\ref{lemma:conjugation}, being established on a dense subspace, extend to the respective essential domains in $\Lint^{2}$.
			Since $W^{(N)}$ is a unitary isomorphism, the identities also carry over to the exponentials.
			Given that $e^{2\pi i u}$ and $\nabla_{v}$ correspond to multiplication operators, checking the relation between $\mhat$ and $\mx$ reduces to
			\begin{equation}
				\mhat W^{(N)} \bigl( f(\xn) \bigr)
				= W^{(N)} \Bigl( e^{2 \pi i \frac{x}{\sqrt{N}}} e^{2 \pi i \frac{n}{N}} e^{2 \pi i \frac{i \qb-1}{\sqrt{N}} x} f (\xn) \Bigr)
				= W^{(N)} \bigl( \mx f(\xn) \bigr) \, .
			\end{equation}
		
			On the other hand, we have that
			\begin{equation}
				\exp \biggl( - \frac{i\qb - 1}{N} \nabla_{u} \biggr) W^{(N)} = W^{(N)} \exp \biggl( -\frac{i \qb - 1}{\sqrt{N}} \dbydt{x} \biggr) \, .
			\end{equation}
			Following Remark~\ref{rem:operators_domain} the exponential on the right-hand side acts, in the appropriate sense, as the shift by $- \frac{i\qb - 1}{\sqrt{N}}$ in $x$.
			We may then conclude that
			\begin{equation}
				\ellhat W^{(N)} \bigl(f(\xn)\bigr)
				= W^{(N)} \biggl( f \Bigl( x - \frac{i \qb -1}{\sqrt{N}} - \frac{1}{\sqrt{N}} , n+1 \Bigr)\biggr)
				= W^{(N)} \bigl( \lx f(\xn) \bigr) \, ,
			\end{equation}
			which was our claim.
		\end{proof}
		
	\section{The annihilator of \texorpdfstring{$\AKJ$}{AKJ}}
		\label{sec:AK}
		
		Throughout this section we will always assume that $N$ is an \emph{odd} positive integer.
		For a fixed $S \in \R$, let $t = N+iS$ and
		\begin{equation}
			\qb = - i e^{2rN} \, ,
			\qquad
			\cb = \frac{i (\qb + \qb^{-1})}{2} \, ,
			\qquad
			q^{-\frac{1}{2}} = -e^{i \pi \frac{qb^{2} + 1}{N}}
		\end{equation}
		as before.
		We then have an action of the algebra $\localg$ from~\eqref{eq:localg} on the space of meromorphic functions on $\ANC$ by
		\begin{equation}
			E \mapsto \lx \, ,
				\qquad
			Q \mapsto \mx \, .
		\end{equation}
		As before, if $f$ is a meromorphic function it makes sense to consider its annihilating left ideals $\opideal$ and $\locideal$ in $\localg$ and $\opalg$, respectively:
		\begin{equation}
			\locideal = \Set{ p \in \localg \colon p(\mx , \lx) f = 0 } \, ,
				\qquad
			\opideal = \locideal \cap \opalg \, .
		\end{equation}
		
		\begin{Definition}
			Let $K$ be an embedded knot in a closed oriented $3$-manifold $M$, $\AKJ$ as in~\cite{AK14a,AM16}.
			We call $\widehat{A}_{q,(M,K)}^{\Cx}$, or the $\widehat{A}^{\Cx}$-polynomial of $(M,K)$, the unique element of $\opideal[J^{(\qb,N)}_{M,K}]$ which, as a polynomial in $E$, has lowest degree and co-prime coefficients in $\Z[q^{\pm\frac{1}{2}}, Q]$.
		\end{Definition}
		
		We shall often drop one or more of the subscripts in $\widehat{A}_{q,(M,K)}^{\Cx}$ where no risk of ambiguity is present.
		Recalling from Section~\ref{sec:alg_setup} the notations for the $A$- and $\widehat{A}$-polynomial of a knot, we are now ready to rephrase Theorem~\ref{thm:main_intro} more precisely.
		
		\begin{Theorem}
			\label{thm:main}
			For $K \subset S^{3}$ the figure-eight knot $4_{1}$ or $5_{2}$, we have
			\begin{equation}
				\widehat{A}_{q,K}^{\Cx} (Q, E) \cdot \bigl(Q - 1\bigr) = \widehat{A}_{q,K}^{\text{nh}} (Q, E) \, .
			\end{equation}
			In the evaluation at $q = 1$ (corresponding to the limit $t \to \infty$), we have that
			\begin{equation}
				\bigl( m^{4} - 1 \bigr) \widehat{A}_{1,K}^{\Cx} ( m^{2} , \ell) = A_{K} ( m , \ell ) \, .
			\end{equation}
		\end{Theorem}
		
		We shall dedicate the rest of the paper to the proof of this statement.
	
	\subsection{The figure-eight knot \texorpdfstring{$4_1$}{4 1}}
		\def\knot{4_{1}}
		
		The formula for $\AKJs[\knot] (\xn) = \AKJs[S^{3},\knot]^{(\qb,N)} (\xn)$ for $\xn \in \AN \subset \ANC$ may be found e.g. in~\cite{AM16}, and it reads
		\begin{equation}
			\label{eq:AKinv}
			\AKJs[\knot] (\xn) = e^{4 \pi i \frac{\cb x}{\sqrt{N}}} \int_{\AN} \frac{ \philog (\xn - \yn) {\Braket{\yn}}^{2}}{\philog(\yn) {\Braket{\xn - \yn}}^{2}} \, \dif \yn\, .
		\end{equation}
		
		We look for operators annihilating $\AKJs[\knot]$ by working on the integrand, which we shall call $\Phi = \Phi(\xn,\yn)$.
		The action of $\mx$ and $\lx$ is well defined on meromorphic functions of $(\xn, \yn)$, and so is that of $\my$ and $\ly$ acting analogously through the variable $\yn$.
		This action may be expressed as a representation of the commutative tensor product $\opalg^{\otimes 2}$, whose formal generators we shall denote $E_{1}, Q_{1}, E_{2}, Q_{2}$.
		It is then immediate to check that
		\begin{equation}
			\begin{aligned}
				\lx \Phi ={}& q \mx^{2} \my^{-2} \bigl(1 + q^{-\frac{1}{2}} \mx^{-1} \my \bigr) \Phi \, , \\
				\ly \Phi ={}& \mx^{-2} \bigl(1 + q^{\frac{1}{2}} \mx^{-1} \my \bigr)^{-1} \bigl(1 + q^{-\frac{1}{2}} \my^{-1} \bigr)^{-1} \Phi \, .
			\end{aligned}
		\end{equation}
		By a simple manipulation, this shows that the annihilator of $\Phi$ in $\localg^{\otimes 2}$ contains
		\begin{equation}
			\begin{aligned}
				g_{1} \coloneqq{}& E_{1}Q_{2}^{2} - q^{\rec{2}} Q_{1}Q_{2} - q Q_{1}^{2} \, , \\
				g_{2} \coloneqq{}& E_{2}Q_{1}Q_{2}^{2} + q^{\rec{2}} \bigl( E_{2}Q_{1}^{2} + E_{2}Q_{1} - q \bigr) Q_{2} + q E_{2}Q_{1}^{2} \, .
			\end{aligned}
		\end{equation}
		With the aid of appropriate software (we used Singular~\cite{DGPS}), one may then run elimination to find an element in this ideal that does not contain the variable $Q_{2}$, namely
		\begin{equation}
			\label{eq:Ahat41_found}
			\begin{aligned}
				P ={}& P_{q} (E_{1},Q_{1},E_{2}) = \\
				={}& q^{3}E_{2}^{2} \bigl( q E_{2}Q_{1}^{2} - 1 ) Q_{1}^{2} E_{1}^{2} \\
				& - \bigl( q^{2}E_{2}Q_{1}^{2} - 1 \bigr) \Bigl( q^{4}E_{2}^{2}Q_{1}^{4} - q^{3}E_{2}^{2}Q_{1}^{3} - q ( q^{2} + 1 ) E_{2}Q_{1}^{2} - qE_{2}Q_{1} + 1 \Bigl) E_{1} \\
				& + q E_{2} \bigl( q^{3}E_{2}Q_{1}^{2} - 1 \bigl) Q_{1}^{2} \, .
			\end{aligned}
		\end{equation}
		We shall not report here the full elimination process, which is rather long, tedious, and computationally heavy, but the reader may verify that
		\begin{equation}
			q^{\frac{9}{2}} Q_{1}^{2} P = q a_{1} g_{1} - a_{2} g_{2} \, ,
		\end{equation}
		where
		\begin{equation}
			\begin{split}
				a_{1} ={} & E_{2} Q_{1} \Bigl( \bigl( q E_{2}Q_{1}^{2} - 1\bigr) \bigl( q^{3} E_{2}Q_{1}^{2} + q E_{2}Q_{1} - 1 \bigr) E_{1} + q^{2} E_{2} \bigl( q^{3} E_{2}Q_{1}^{2} - 1 \bigr) Q_{1}^{2} \Bigr) Q_{2} \\
				+ & q^{\frac{1}{2}} \bigl( q E_{2}Q_{1}^{2} - 1 \bigr)
				\Bigl( q^{5} E_{2}^{2}Q_{1}^{4} + q^{3} E_{2}^{2}Q_{1}^{3} + q E_{2}^{2}Q_{1}^{2} - q^{2} ( q + 1) E_{2}Q_{1}^{2} - ( q + 1) E_{2}Q_{1} + 1 \Bigr) E_{1} \\
				+ & q^{\frac{5}{2}} E_{2} \bigl( E_{2}Q_{1} - q \bigr) \bigl(q^{3} E_{2}Q_{1}^{2} - 1 \bigr) Q_{1}^{2}
			\end{split}
		\end{equation}
		and
		\begin{equation}
			\begin{split}
				a_{2} ={}& \biggl( \bigl( q E_{2}Q_{1}^{2} - 1 \bigr) \Bigl( q^{3} E_{2}Q_{1}^{2} + q E_{2}Q_{1} - 1 \Big) E_{1}^{2} + q^{3} E_{2} \Bigl(q^{3} E_{2}Q_{1}^{2} - 1 \Bigr) Q_{1}^{2} E_{1} \biggr) Q_{2} \\
				- & q^{\frac{5}{2}} E_{2} \bigl( q E_{2} Q_{1}^{2} - 1 \bigr) Q_{1}^{2}E_{1}^{2} \\
				- & q^{\frac{3}{2}} \Bigl( q^{4} ( q + 1 ) E_{2}^{2}Q_{1}^{4} + q^{2} E_{2}^{2}Q_{1}^{3} - q ( q^{2} + q + 1 ) E_{2}Q_{1}^{2} - q E_{2}Q_{1} + 1 \Bigr) Q_{1} E_{1} \\
				- & q^{\frac{7}{2}} E_{2} \bigl( q^{3} E_{2}Q_{1}^{2} - 1 \bigr) Q_{1}^{3} \, .
			\end{split}
		\end{equation}
		
		In order to obtain from $P \in \opalg^{\otimes 2}$ an element of $\opideal[{\AKJs[\knot]}]$, we need to show that, in an appropriate sense, all monomials in $\mx$, $\lx$, and $\ly$ can be taken out of the integral.
		While this is clearly the case for $\mx$, convergence of the integral does depend on the value of $\xn$ and on the specific contour, and some care is needed when shifting either variable.
		With that in mind, for fixed $h < 0$ and $a \in - T - \frac{\cb}{\sqrt{N}}$ we define (cf. Figure~\ref{fig:convergence}) a region
		\begin{equation}
			R_{h, a} \coloneqq \Set{ \xi + \lambda \frac{i\qb}{\sqrt{N}} \in T + a \colon \xi \in \R \text{ and } \frac{h}{2} < \lambda < -h } \subset \Cx
		\end{equation}
		and a contour
		\begin{equation}
			\label{eq:gamma_ah}
			\gamma_{h, a} \coloneqq \partial \Biggl( \Set{y \in \Cx \colon \Im \biggl(y - h \frac{i\qb}{\sqrt{N}} \biggr) \geq 0} \setminus \biggl(T + \frac{\cb}{\sqrt{N}} + a \biggr) \Biggr) \times \Z \slash N\Z \subset \ANC \, ,
		\end{equation}
		where we recall that $T$ is as in~\eqref{eq:T}.
				
		\begin{figure}
			\centering
				\subfloat[][\label{fig:poles}
					For fixed $x$, the poles of the integrand occur inside the infinite triangles with tips at $x$ and $-\cb$.
					The shifted triangle with tip at $a$ contains all these points if it contains both $x$ and $-\cb$.]{
				\begin{tikzpicture}[auto,scale=.16,cm={1,0,0,1,(0,0)},>=latex]
					\tikzstyle{connector}=[thick]
					\def\point{8pt};
					\def\xend{15};
					\def\yupend{11};
					\def\yend{11};
					\def\Reb{3.1};
					\def\Imb{1.7};
					\def\Rex{-6.2};
					\def\Imx{4.4};
					\def\eps{2.3};
					
					\draw[thick,->] (-\xend,0) -- (\xend,0);
						\draw (\xend,0) coordinate[label=below left:{$\R$}] (R);
					\draw[thick,->] (0,-\yend) -- (0,\yupend);
						\draw (0,\yupend) coordinate[label=below right:{$i\R$}] (iR);
					
					\draw (\Reb,\Imb) 							coordinate[label=above right:{$\qb$}] (qb);
					\draw (\Reb,-\Imb)							coordinate (barb);
					\draw (-\Imb,\Reb)							coordinate (ib);
					\draw (\Imb,\Reb)							coordinate (ibarb);
					\draw (0,\Reb)								coordinate (cb);
					\draw (0,-\Reb)								coordinate[label=right:{$-\cb$}] (mcb);
					\draw (\Rex,\Imx)							coordinate[label=above left:{$x$}] (x);
					\draw ($(x)-(cb)+{-((\Imx-\Reb)+\yend)/\Reb}*(ibarb)$)	coordinate[label=above right:{$x + T$}] (blx);
					\draw ($(x)-(cb)+{-((\Imx-\Reb)+\yend)/\Reb}*(ib)$)		coordinate (brx);
					\draw ($(mcb)+{-(\yend-\Reb)/\Reb}*(ibarb)$)				coordinate (bl);
					\draw ($(mcb)+{-(\yend-\Reb)/\Reb}*(ib)$)						coordinate[label=above right:{$T$}] (br);
					
			
				\fill[color=gray,opacity=.2] ($(x)-(cb)$) -- ($(blx)$) -- ($(brx)$) -- cycle;
				\fill[color=gray,opacity=.2] ($(mcb)$) -- ($(bl)$) -- ($(br)$) -- cycle;
								
					\fill (qb) circle[radius=\point];
					\fill (x)	circle[radius=\point];
					
					\pgfmathsetmacro{\len}{int(\yend/\Reb)-1}
					\foreach		\alpha	in {0,...,\len}				{
						\pgfmathsetmacro{\end}{\len-\alpha}
						\foreach	\beta		in {0,...,\end}			{
							\draw	(${-1}*(cb)+{-\alpha}*(ib)+{-\beta}*(ibarb)$)		coordinate
								(p\alpha\beta);
							\fill[color=gray] 		(p\alpha\beta)												circle[radius=\point];
							}
						}
						
					\pgfmathsetmacro{\len}{int((\yend+\Imx)/\Reb)-1}
					\foreach		\alpha	in {0,...,\len}				{
						\pgfmathsetmacro{\end}{\len-\alpha}
						\foreach	\beta		in {0,...,\end}			{
							\draw	($(x)+{-1}*(cb)+{-\alpha}*(ib)+{-\beta}*(ibarb)$)		coordinate
								(p\alpha\beta);
							\draw 		(p\alpha\beta)												circle[radius=\point];
							}
						}
   
				\end{tikzpicture}
			}
			\qquad
			\subfloat[][\label{fig:contour}
				The contour follows $\R \!+\! i h \qb$ ($h < 0$) and deviates along the triangle with tip at $a$.
				The integrand decays quickly in $y$ if $x$ lies in the strip, and the poles lie below $\gamma_{h,a}$ if $-\cb$ and $x$ do.]{
				\begin{tikzpicture}[auto,scale=.16,cm={1,0,0,1,(0,0)},>=latex]
					\tikzstyle{connector}=[thick]
					\def\point{8pt};
					\def\xend{15};
					\def\yupend{11};
					\def\yend{11};
					\def\Reb{4};
					\def\Imb{2};
					\def\Rea{-2};
					\def\Ima{7.5};
					\def\eps{1.4};
					
					\draw[thick,->] (-\xend,0) -- (\xend,0);
						\draw (\xend,0) coordinate[label=below left:{$\R$}] (R);
					\draw[thick,->] (0,-\yend) -- (0,\yupend);
						\draw (0,\yupend) coordinate[label=below right:{$i\R$}] (iR);
					
					\draw (\Reb,\Imb) 							coordinate[label=above:{$\qb$}] (qb);
					\draw (\Reb,-\Imb)							coordinate (barb);
					\draw (-\Imb,\Reb)							coordinate (ib);
					\draw (\Imb,\Reb)							coordinate (ibarb);
					\draw (0,\Reb)								coordinate (cb);
					\draw (0,-\Reb)								coordinate[label=below left:{$-\cb$}] (mcb);
					\draw (\Rea,\Ima)					coordinate[label=above left:{$a$}] (a);
					\draw ($(a)+{-{\Ima/\Reb}-\eps}*(ibarb)$)			coordinate (cross1);
					\draw ($(a)+{-{\Ima/\Reb}-\eps}*(ib)$)				coordinate (cross2);
					\draw ($(a)+{-{\Ima/\Reb}+\eps}*(ibarb)$)			coordinate (par1);
					\draw ($(a)+{-{\Ima/\Reb}+\eps}*(ib)$)				coordinate (par2);
					\draw ($(a)+{-(\Ima/\Reb)-((.5)*\eps)}*(ib)$)		coordinate (par3);
					\draw ($(a)+{-(\Ima/\Reb)-((.5)*\eps)}*(ibarb)$)	coordinate (par4);
							
					\fill[color=gray,opacity=.2] (-\xend,-.5*\eps*\Reb) -- (\xend,-.5*\eps*\Reb) -- (\xend,\eps*\Reb) -- (-\xend,\eps*\Reb) -- cycle;
					\fill[color=gray,opacity=.2]	($(-\xend,0)+{-\eps}*(cb)$) -- node[below] {$\gamma_{\varepsilon,a}$} ($(cross1)$) -- ($(a)$) -- ($(cross2)$) -- ($(\xend,0)+{-\eps}*(cb)$) -- (\xend,-\yend) -- (-\xend,-\yend) -- cycle;
					\fill[color=gray,opacity=.2] ($(par1)$) -- ($(par2)$) -- ($(par3)$) -- ($(par4)$) -- cycle;
						
					\draw ($(-\xend,0)+{\eps}*(cb)$) -- ($(\xend,0)+{\eps}*(cb)$);
					\draw ($(-\xend,0)+{-{.5}*\eps}*(cb)$) -- ($(\xend,0)+{-{.5}*\eps}*(cb)$);
					\draw[ultra thick] ($(-\xend,0)+{-\eps}*(cb)$) -- node[below] {$\gamma_{\rho,a}$} ($(cross1)$) -- ($(a)$) -- ($(cross2)$) -- ($(\xend,0)+{-\eps}*(cb)$);
					
					\fill (qb) circle[radius=\point];
					\fill (0,-\Reb)	circle[radius=\point];
					\draw (\Rea,1.2)	node		(R)	{$R_{h,a}$};
			
				\end{tikzpicture}
			}
			\caption{The distribution of the poles of the integrand and the contour are illustrated for $N=1$. The situation is analogous for higher $N$, up to rescaling $\cb$ by $\sqrt{N}$ and replicating the picture $N$ times.}
			\label{fig:convergence}
		\end{figure}
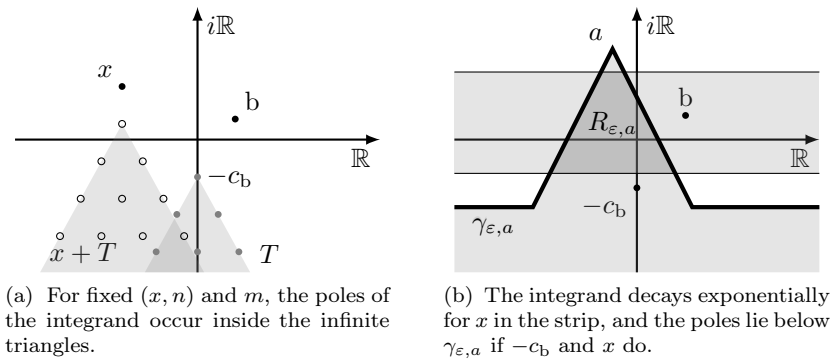
		
		\begin{Proposition}
			\label{prop:conv41}
			Suppose $h < 0$ and $a \in - T - \frac{\cb}{\sqrt{N}}$ are fixed.
			For every $\xn \in \ANC$ with $x \in R_{h,a}$, the integral
			\begin{equation}
				\chi_{h, a} (\xn) \coloneqq \int_{\gamma_{h, a}} \Phi(\xn, \yn) \dif \yn
			\end{equation}
			is absolutely convergent.
			The function $\chi_{h, a}$ is holomorphic, and if $x \in R_{h, a} \cap \R$ then
			\begin{equation}
				e^{4 \pi i \frac{\cb x}{\sqrt{N}}} \chi_{h, a} (\xn) = \AKJs[\knot] (\xn) \, .
			\end{equation}
		\end{Proposition}
		
		\begin{proof}
			For fixed $\xn$, the singularities of $\Phi(\xn, \yn)$ lie at the zeroes of $\philog(\yn)$ and poles of $\philog(\xn - \yn)$.
			These occur for $y$ in $T$ and $T+x$, respectively, both of which are contained in $T + a$ if $a \in -T - \frac{\cb}{\sqrt{N}}$ and $x \in T + a$.
			Therefore, under these conditions, the contour $\gamma_{h,a}$ avoids all the singularities of the integrand.
			
			In order to study the behaviour of $\Phi$ at infinity, express $x$ and $y$ as $\xi + \lambda \frac{i\qb}{\sqrt{N}}$ and $\eta + \rho \frac{i\qb}{\sqrt{N}}$, respectively---such expressions exist uniquely since $\Re(\qb) > 0$.
			Using Lemma~\ref{lemma:dilog_asymptotics} and expanding the definition of $\braket{\yn}$ and $\braket{\xn-\yn}$~\eqref{eq:FourierGaussian} we see that
			\begin{equation}
				\label{eq:asymptotic_41}
				\abs{\Phi(\xn,\yn)} \approx
				\begin{cases}
					\abs{\displaystyle \frac{\braket{\yn}^{2}}{\braket{\xn-\yn}}}
						= C_{-} (\xn, \rho) e^{-2\pi \eta (\rho + \lambda) \frac{\Re(\qb)}{\sqrt{N}}}
						& \text{for $\eta \to - \infty$} \\[1em]
					\abs{\displaystyle \frac{\braket{\yn}}{\braket{\xn-\yn}^{2}}}
						= C_{+} (\xn, \rho) e^{2\pi \eta (\rho - 2 \lambda) \frac{\Re(\qb)}{\sqrt{N}}}
						& \text{for $\eta \to + \infty$}
				\end{cases}
			\end{equation}
			for some continuous functions $C_{\pm}$ of $\xn$ and $\rho$ alone.
			We then see that, for $\rho = h$, the integrand decays exponentially near $- \infty$ if $\lambda < - \rho$ and near $+ \infty$ if $\lambda > \frac{\rho}{2}$.
			When that is the case, the integral is absolutely convergent, which proves the first part of our statement.
			
			For the second part, suppose that $\xn_{0} = (x_{0}, n_{0})$ is fixed with $x_{0} \in R_{h,a}$, and that $B \subset R_{h, a}$ is a compact neighbourhood of $x_{0}$.
			By~\eqref{eq:asymptotic_41} it is then easy to bound $\abs{\Phi(\xn,\yn)}$ by an absolutely integrable function of $\yn$ alone, uniformly for all $x \in B$, whence the continuity of $\chi_{h,a}$ on that region by dominated convergence.
			Furthermore, for any closed contour $\alpha$ inside $B$, the same bound justifies the use of Fubini-Tonelli in
			\begin{equation}
				\int_{\alpha} \chi_{h, a} (\xn) \dif x = \int_{\gamma_{h,a}} \int_{\alpha} \Phi(\xn,\yn) \dif x \dif \yn \, .
			\end{equation}
			The holomorphicity of $\chi_{h,a}$ follows then from Morera's theorem.
			
			For the final part, suppose that $\xn$ is fixed, with $x \in R_{h,a} \cap \R$, which is to say that $\lambda = 0$.
			For a positive real number $M$ consider, for each component in $\ANC$, the compact region $D$ enclosed by $\AN$, $\gamma_{h,a}$, and the lines $i\qb \R \pm M$.
			Since $\Phi(\xn,\yn)$ is holomorphic in $y$ on $D$, the integral of $\Phi(\xn,\yn) \dif y$ around $\partial D$ vanishes.
			Using~\eqref{eq:asymptotic_41} again, the integrand is bounded by $2 C_{\pm} (\xn,\rho)$ on the components of $\partial D$ along $i \qb \R \pm M$, uniformly in $M$.
			It is then easy to see that the corresponding contributions vanish in the limit for $M \to \infty$, showing that
			\begin{equation}
				\int_{\AN} \Phi(\xn, \yn) \dif \yn = \int_{\gamma_{h,a}} \Phi(\xn,\yn) \dif \yn \, ,
			\end{equation}
			which concludes our proof.
		\end{proof}
		
		The proposition vindicates the claim that $\AKJs[\knot]$ extends holomorphically to $\ANC$.
		It is now also clear that, for $h$ and $a$ as usual, $h' = h-1$, $a' = a - \frac{i\qb}{\sqrt{N}}$, and $x \in R_{h',a'}$, we have
		\begin{equation}
			e^{4 \pi i \frac{\cb x}{\sqrt{N}}} \int_{\gamma_{h,a}} \ly \Phi(\xn,\yn) \dif \yn
			= e^{4 \pi i \frac{\cb x}{\sqrt{N}}} \int_{\gamma_{h',a'}} \Phi(\xn,\yn) \dif \yn 
			= \AKJs[\knot] (\xn) \, ,
		\end{equation}
		and that if both $x, x - \frac{i\qb}{\sqrt{N}} \in R_{h,a}$ then
		\begin{equation}
			e^{4 \pi i \frac{\cb x}{\sqrt{N}}} \int_{\gamma_{h,a}} \lx \Phi (\xn,\yn) \dif \yn = q^{-1} \lx \AKJs[\knot](\xn) \, .
		\end{equation}
		Up to choosing $h$ and $a$ sufficiently large, we see then that
		\begin{equation}
			0 = e^{4 \pi i \frac{\cb x}{\sqrt{N}}} \int_{\gamma_{h,a}} P_{q} \bigl(\lx,\mx,\ly \bigr) \Phi (\xn,\yn) \dif \yn = P_{q} \bigl(q^{-1}\lx,\mx,1 \bigr) \AKJs[\knot] (\xn)
		\end{equation}
		on some open subset of $\ANC$, and therefore $P_{q} (q^{-1}\lx,\mx,1)$ annihilates $\AKJs[\knot]$.
		
		We are now ready to prove Conjecture~\ref{conj:complexAJ} for $\knot$.
		
		\begin{Theorem}
			Conjecture~\ref{conj:complexAJ} holds for the figure-eight knot.
		\end{Theorem}
		
		\begin{proof}
			Call $P' = P'_{q} (E,Q) \coloneqq P_{q} (q^{-1} E, Q, 1)$, so that $P'_{q}(\lx, \mx) \AKJs[\knot] = 0$.
			We can see by direct comparison of~\eqref{eq:Ahat41_found} with~\eqref{eq:Ahat41_known} and~\eqref{eq:clA41} that
			\begin{equation}
				q P'_{q} (E, Q) (Q-1) = \widehat{A}_{q,4_{1}}^{\text{nh}} (E, Q)
				\qquad \text{and} \qquad
				P'_{1}(\ell, m^{2}) = \bigl(m^{4} -1 \bigr) A_{\knot} (\ell, m) \, .
			\end{equation}
			By definition, $\widehat{A}_{q,\knot}^{\Cx}$ is the preferred generator of $\locideal[\knot]$, and therefore
			\begin{equation}
				P' = p \widehat{A}^{\Cx}_{q,\knot}
			\end{equation}
			for some $p \in \localg$.
			In the evaluation at $q = 1$, the above gives a factorisation of $(m^{4} -1) A_{\knot}$, and since $A_{\knot}$ is known to be irreducible it follows that only one between $p$ and $\widehat{A}_{q,\knot}^{\Cx}$ can contain the variable $E$.
			On the other hand, if $\widehat{A}_{q,\knot}^{\Cx}$ were a polynomial of $Q$ alone it would follow that $\AKJs[\knot] = 0$, a contradiction.
			Therefore, $p$ is a non-zero polynomial in $Q$, so we can write $\widehat{A}^{\Cx}_{q,\knot} = p^{-1} P'$ in $\localg$.
			Since $q P'$ has integer and co-prime coefficients, it follows that $p = q^{-1}$ and $\widehat{A}^{\Cx}_{q,\knot} = q P'$, which as we have seen satisfies all the claimed properties.
		\end{proof}
		
	\subsection{The knot \texorpdfstring{$5_2$}{5 2}}
		\def\knot{5_2}
		\def\AKinv{\chi_{5_2}}

		The discussion for $\knot$ is similar.
		From~\cite{AM16} we have
		\begin{equation}
			\label{eq:AK52}
			\AKJs[\knot] (\xn) = e^{2 \pi i \frac{\cb x}{\sqrt{N}}} \int_{\AN} \frac{\Braket{\yn} \inv{\Braket{\xn}}}{\philog(\yn + \xn) \philog(\yn) \philog(\yn-\xn)} \, \dif \yn \, .
		\end{equation}
		We will again call $\Phi = \Phi(\xn,\yn)$ the integrand and see that
		\begin{equation}
			\begin{aligned}
				\lx \Phi ={}& q^{\frac{1}{2}} \mx \inv{\bigl( 1 + q^{-\rec{2}} \mx^{-1} \my^{-1} \bigr)} \bigl( 1 + q^{\rec{2}} \mx \my^{-1} \bigr) \Phi \, , \\
				\ly \Phi ={}& q^{-\rec{2}} \my^{-1} \inv{ \bigl( 1 + q^{- \rec{2}} \mx^{-1} \my^{-1} \bigr)} \inv{ \bigl( 1 + q^{-\rec{2}} \my^{-1} \bigr)} \inv{ \bigl( 1 + q^{-\rec{2}} \my^{-1} \mx \bigr)} \Phi \, .
			\end{aligned}
		\end{equation}
		Therefore, the annihilator of $\Phi$ in $\localg^{\otimes 2}$ contains
		\begin{equation}
			\begin{aligned}
				g_{1} \coloneqq{}& q^{\rec{2}} \bigl( Q_{1}E_{1} - q^{\frac{1}{2}} Q_{1}^{2} \bigr) Q_{2} + E_{1} - q^{\frac{3}{2}} Q_{1}^{3} \, , \\
				g_{2} \coloneqq{}& E_{2}Q_{1}Q_{2}^{3} + q^{\rec{2}} \Bigl( E_{2}Q_{1}^{2} + E_{2}Q_{1} - q^{2} Q_{1} + E_{2} \Bigr) Q_{2}^{2}
									+ q E_{2} \Bigl( Q_{1}^{2} + Q_{1} + 1 \Bigr) Q_{2} + q^{\frac{3}{2}} E_{2}Q_{1} \, .
			\end{aligned}
		\end{equation}
		By eliminating $E_{2}$ we find that the element
		\begin{equation}
			\begin{split}
				P ={}& P_{q} (E_{1}, Q_{1}, E_{2}) \\ \coloneqq{}
					& - q^{\rec{2}} \bigl( q Q_{1}^{2} - 1 \bigr) \bigl( q^{2} Q_{1}^2 - 1 \bigr) E_{1}^{3} \\
					& + q \bigl( q Q_{1}^{2} - 1 \bigr) \bigl( q^{4} Q_{1}^{2} - 1 \bigr) \Bigr( q^{9} E_{2}Q_{1}^{5} - q^{7} E_{2}Q_{1}^{4} - q^{4} ( q^{3} + 1) E_{2}Q_{1}^{3} \\
					& \quad + q^{5}(q + 1) Q_{1}^{3} + q^{2} (q^{3} + 1) E_{2}Q_{1}^{2} + q^{2} ( E_{2} + 1 ) Q_{1} - E_{2} \Bigr) E_{1}^{2} \\
					& +  q^{\frac{9}{2}} Q_{1}^{2} \bigl( q^{2} Q_{1}^{2} - 1 \bigr) \bigl( q^{5} Q_{1}^{2} - 1 \bigr) \Bigl( q^{6} E_{2}Q_{1}^{5} - q^{5} ( E_{2} + 1) Q_{1}^{4} \\
					& \quad - q^{2} ( q^{3} + 1) E_{2}Q_{1}^{3} + q (q^{3} E_{2} - q^{2} - q + E_{2}) E_{2}Q_{1}^{2} + q E_{2}Q_{1} - E_{2} \Big) E_{1} \\
					& + q^{8} Q_{1}^{7} \bigl( q^{4} Q_{1}^{2} - 1 \bigr) \bigl( q^{5} Q_{1}^{2} - 1 \bigr)
			\end{split}
		\end{equation}
		may be expressed as
		\begin{equation}
			P = a_{1} g_{1} + Q_{1} g_{2}
		\end{equation}
		with
		\begin{equation}
			\begin{split}
				a_{1} ={}&
				\Bigl( -q^{\rec{2}} E_{2} Q_{1} \bigl( q Q_{1}^{2} - 1 \bigr) \bigl( q^{2} Q_{1}^{2} - 1 \bigr) E_{1}^{2} + q^{2} (q + 1) E_{2} Q_{1}^{2} \bigl( q Q_{1}^{2} - 1 \bigr) \bigl( q^{5} Q_{1}^{2} - 1 \bigr) E_{1} \\
				& \quad - q^{\frac{7}{2}} E_{2} Q_{1}^{3} \bigl( q^{4} Q_{1}^{2} - 1 \bigr) \bigl( q^{5} Q_{1}^{2} - 1 \bigr) \Bigr) Q_{2}^{2} \\
				& + \Bigl( - q Q_{1} \bigl( q Q_{1}^{2} - 1 \bigr) \bigl( q^{2} Q_{1} - 1 \bigr) \bigl(q^{3} E_{2}Q_{1} + E_{2} - q^{2} \bigr) E_{1}^{2} \\
				& \quad + q^{\frac{5}{2}} Q_{1} \bigl( q Q_{1}^{2} - 1 \bigl) \bigl( q^{5} Q_{1}^{2} - 1 \bigr)  \Bigl( q^{3} E_{2}Q_{1}^{2} + (q + 1) E_{2}Q_{1} - q^{2} ( q + 1) Q_{1} + E_{2} \Bigr) E_{1} \\
				& \quad - q^{8} Q_{1}^{2} \bigl( q^{4} Q_{1}^{2} - 1 \bigr) \bigl( q^{5} Q_{1}^{2} - 1 \bigr) \bigl( E_{2}Q_{1} - q^{2} Q_{1} + E_{2} \bigl) \Bigr) Q_{2} \\
				& - q^{\frac{1}{2}} \bigl( q Q_{1}^{2} - 1 \bigr) \bigl( q^{2} Q_{1}^{2} - 1 \bigr) \bigl( q^{4} E_{2}Q_{1}^{2} + 1 \bigr) E_{1}^{2} - q  \bigl( q Q_{1}^{2} - 1 \bigr) \bigl( q^{5} Q_{1}^{2} - 1 \bigr) \\
				& \quad \cdot \Bigl( q^{6} E_{2}Q_{1}^{4} - q^{5} E_{2}Q_{1}^{3} - q^{5} Q_{1}^{3} - q^{2} ( q^{2} + 1 ) E_{2}Q_{1}^{2} - q^{2} E_{2}Q_{1} - q^{2} Q_{1} + E_{2} \Bigr)E_{1} \\
				& - q^{\frac{9}{2}} Q_{1}^{2} \bigl( q^{2} Q_{1}^{2} + E_{2} \bigr) \bigl( q^{4} Q_{1}^{2} - 1 \bigr) \bigl( q^{5} Q_{1}^{2} - 1 \bigr)
			\end{split}
		\end{equation}
		and
		\begin{equation}
			\begin{split}
				a_{2} ={}& 
				\bigl( q Q_{1}^{2} - 1 \bigr) \bigl( q^{2} Q_{1}^{2} - 1 \bigr) E_{1}^{3} - q^{\frac{3}{2}} Q_{1} \bigl( q^{2} + q + 1 \bigr) \bigl( q Q_{1}^{2} - 1 \bigr) \bigl( q^{4} Q_{1}^{2} - 1 \bigr) E_{1}^{2} \\
				& + q^{3} Q_{1}^{2} \bigl( q^{2} + q + 1 \bigr) \bigl( q^{2} Q_{1}^{2} - 1 \bigr) \bigl( q^{5} Q_{1}^{2} - 1 \bigr) E_{1} \\
				& - q^{\frac{9}{2}} Q_{1}^{3} \bigl( q^{4} Q_{1}^{2} - 1 \bigr) \bigl( q^{5} Q_{1}^{2} - 1 \bigr) \, .
			\end{split}
		\end{equation}
		
		\begin{Proposition}
			Let $h < 0$ and $a \in -T - \frac{\cb}{\sqrt{N}}$ be fixed, $\gamma_{h,a}$ as in~\eqref{eq:gamma_ah}, and
			\begin{equation}
				R_{a} \coloneqq \Set{ (x,n) \in \ANC \colon x \in \Bigl( T + \frac{\cb}{\sqrt{N}} + a \Bigr) \cap \Bigl( -T - \frac{\cb}{\sqrt{N}} -a \Bigr) } \, .
			\end{equation}
			For every $\xn \in R_{a}$, the integral
			\begin{equation}
				\chi_{h,a} (\xn) \coloneqq \int_{\gamma_{h,a}} \Phi(\xn,\yn) \dif \yn
			\end{equation}
			is absolutely convergent.
			The function $\chi_{h,a}$ is holomorphic, and if $\xn \in R_{a} \cap \AN$ then
			\begin{equation}
				e^{2 \pi i \frac{\cb x}{\sqrt{N}}} \chi_{h,a} (\xn) = \AKJs[\knot] (\xn) \, .
			\end{equation}
		\end{Proposition}
		
		\begin{proof}
			For fixed $\xn$, every pole of $\Phi(\xn,\yn)$ has $y \in (T-x) \cup T \cup (T+x)$.
			A simple check shows that
			\begin{equation}
				x \in \pm \Bigl(T+\frac{\cb}{\sqrt{N}}+a \Bigr) \implies T \pm x \subset T + \frac{\cb}{\sqrt{N}} + a \, ,
			\end{equation}
			respectively.
			Therefore, if $\xn \in R_{a}$ and $a \in - T - \frac{\cb}{\sqrt{N}}$, then all the poles of $\Phi(\xn, \yn)$ lie inside $T + \frac{\cb}{\sqrt{N}} + a$, and in particular strictly below $\gamma_{h,a}$.
			
			Writing $x = \xi + \lambda \frac{i \qb}{\sqrt{N}}$ and $y = \eta + \rho \frac{i \qb}{\sqrt{N}}$, and using Lemma~\ref{lemma:dilog_asymptotics}, we see that
			\begin{equation}
				\abs{\Phi(\xn, \yn)} \approx
				\begin{cases}
					\abs{\braket{\yn} \braket{\xn}^{-1}} = C_{-} (\xn, \rho) e^{-2 \pi \eta \rho \frac{\Re(\qb)}{\sqrt{N}}}
					& \text{for } \eta \to -\infty	\\[.8em]
					\abs{\displaystyle \frac{\braket{\xn}^{-1}}{\braket{\yn+\xn} \braket{\yn-\xn}}}
					= \abs{\displaystyle \frac{1}{\braket{\yn}^{2} \braket{\xn}^{3}}}
					= C_{+} (\xn, \rho) e^{4\pi \eta \rho \frac{\Re(\qb)}{\sqrt{N}}}
					& \text{for } \eta \to +\infty
				\end{cases}
			\end{equation}
			for appropriate continuous functions $C_{\pm}$.
			Therefore, $\Phi$ decays exponentially at infinity along $\gamma_{h,a}$ as long as $h < 0$, regardless of the value of $\xn$, establishing absolute convergence of the integral.
			
			The rest of the proof is essentially identical to that of Proposition~\ref{prop:conv41}.
		\end{proof}
		
		As in the case of $4_{1}$, we may conclude that each $\chi_{h,a}$ is the holomorphic extension of $\AKJs[\knot]$ on $R_{a}$, and that
		\begin{equation}
			e^{2 \pi i \frac{\cb x}{\sqrt{N}}} \int_{\gamma_{h,a}} \ly^{d} \Phi(\xn,\yn) \dif \yn = \AKJs[\knot] (\xn)
		\end{equation}
		for every $d \in \Z_{\geq 0}$ provided that $h < -d$ and $\xn \in R_{a}$, and that
		\begin{equation}
			e^{2 \pi i \frac{\cb x}{\sqrt{N}}} \int_{\gamma_{h,a}} \lx \Phi(\xn,\yn) \dif \yn
			= e^{2 \pi i \frac{\cb x}{\sqrt{N}}} \lx \chi_{h,a} 
			= - q^{-\frac{1}{2}} \lx \AKJs[\knot] (\xn)
		\end{equation}
		if $\xn$ and $\xn + \frac{i\qb}{\sqrt{N}}$ lie in $R_{a}$.
		Choosing $h$ and $a$ appropriately, we conclude that
		\begin{equation}
			0 = \int_{\gamma_{a,h}} P_{q} \bigl(\lx, \mx, \ly\bigr) \Phi(\xn, \yn) \dif \yn = P_{q} \bigl( - q^{-\frac{1}{2}} \lx, \mx, 1\bigr) \AKJs[\knot] (\xn)
		\end{equation}
		on some open subset, so $P_{q} \bigl(-q^{-\frac{1}{2}} E, Q, 1\bigr) \in \opideal[{\AKJs[\knot]}]$.
		
		\begin{Theorem}
			Conjecture~\ref{conj:complexAJ} holds for the knot $5_{2}$.
		\end{Theorem}
		
		\begin{proof}
			Calling $P_{q}' (E, Q) \coloneqq P_{q} \bigl(-q^{-\frac{1}{2}} E, Q, 1 \bigr)$, we see that $P_{q}' \bigl( \lx, \mx \bigr) \AKJs[\knot]$ and
			\begin{equation}
				q P_{q}' (E,Q) (Q-1) = \widehat{A}^{\text{nh}}_{q, \knot} (E,Q)
				\qquad \text{and} \qquad
				P'_{1} (\ell, m^{2}) = (m^{4}-1)^{2} A_{\knot} (\ell, m) \, .
			\end{equation}
			The conclusion that $\widehat{A}_{q,\knot}^{\Cx} = q P'$ follows by the same argument as for $4_{1}$.
		\end{proof}
		

\todos

\end{document}